\documentclass[11pt,review]{elsarticle}

\usepackage{lineno,hyperref}
\modulolinenumbers[5]
\usepackage{amssymb}
\usepackage{moreverb,multirow}
\usepackage{amssymb}
\usepackage{amsmath,amsthm}
\usepackage{amsfonts,dsfont,mathtools}
\usepackage[margin=2cm]{geometry}
\usepackage[hypcap]{caption}
\usepackage{caption}
\usepackage{subcaption}
\usepackage{multicol}
\usepackage{morefloats}
\usepackage{booktabs,tabularx}
\usepackage{multirow}
\usepackage{url}
\usepackage{color,xcolor}
\usepackage{graphics, graphicx, epstopdf}
\usepackage{algorithm}
\theoremstyle{plain}
\newtheorem{theorem}{Theorem}[section]

\theoremstyle{definition}

\theoremstyle{remark}

\theoremstyle{remark}

\theoremstyle {plain}

\usepackage{lineno,hyperref}

\journal{xxxxx}

\begin{document}
	

\begin{frontmatter}

%
%
%

\title{Optimal control and comprehensive cost-effectiveness analysis  for COVID-19}


\author[mymainaddressa,mymainaddressb]{Joshua Kiddy Kwasi Asamoah\corref{mycorrespondingauthor}}
\ead{topeljoshua@gmail.com}
\address[mymainaddressa]{Complex Systems Research Center, Shanxi University,
	Taiyuan 030006, PR China}
\address[mymainaddressb]{Shanxi Key Laboratory of Mathematical Techniques and Big Data Analysis on Disease Control and Prevention, Shanxi University, Taiyuan 030006, PR China}
\author[mymainaddressEO]{Eric Okyere}
\address[mymainaddressEO]{Department of Mathematics and Statistics, University of Energy and Natural Resources, Sunyani, Ghana}
\author[mymainaddressbKKK]{Afeez Abidemi}
\address[mymainaddressbKKK]{Department of Mathematical Sciences, Federal University of Technology Akure, PMB 704, Ondo State, Nigeria}
\author[mysecondaryaddress2]{Stephen E. Moore}
\address[mysecondaryaddress2]{Department of Mathematics, University of Cape Coast, Cape Coast, Ghana}
\author[mymainaddressa,mymainaddressbc]{ Gui-Quan Sun\corref{mycorrespondingauthor}}
\ead{sunguiquan@sxu.edu.cn}
\address[mymainaddressbc]{Department of Mathematics, North University of China, Taiyuan, Shanxi 030051, China}
\cortext[mycorrespondingauthor]{Corresponding authors}
\author[mymainaddressa,mymainaddressb]{Zhen Jin }
\author[mysecondaryaddress2]{Edward Acheampong}
\address[mysecondaryaddress2]{Department of Statistics and Actuarial Science University of Ghana,
P.O. Box LG 115, Legon, Ghana}
\author[mymainaddressk]{Joseph Frank Gordon}
\address[mymainaddressk]{Department of Mathematics Education, Akenten Appiah Menka University of Skills Training and Entrepreneurial Development, Kumasi, Ghana}

\begin{abstract}
Cost-effectiveness analysis is a mode of determining both the cost and economic health outcomes of one or more control interventions. In this work, we have formulated a non-autonomous nonlinear deterministic model to study the control of COVID-19 to unravel the cost and economic health outcomes for the autonomous nonlinear model proposed for the Kingdom of Saudi Arabia.
 The optimal control model captures four time-dependent control functions, thus, $u_1$-practising physical or social distancing protocols; $u_2$-practising personal hygiene by cleaning contaminated surfaces with alcohol-based detergents; $u_3$-practising proper and safety measures by exposed, asymptomatic and symptomatic infected individuals; $u_4$-fumigating schools in all levels of education, sports facilities, commercial areas and religious worship centres. We proved the existence of the proposed optimal control model. The optimality system associated with the non-autonomous epidemic model is derived using Pontryagin's maximum principle. We have performed numerical simulations to investigate extensive cost-effectiveness analysis for fourteen optimal control strategies. Comparing the control strategies, we noticed that; Strategy 1 (practising physical or social distancing protocols) is the most cost-saving and most effective control intervention in Saudi Arabia in the absence of vaccination. But, in terms of the infection averted, we saw that strategy 6, strategy 11, strategy 12, and strategy 14 are just as good in controlling COVID-19.
\end{abstract}
\end{frontmatter}

\section{Introduction}
The recent worldwide outbreaks of COVID-19 infectious disease has attracted a lot of attention in the mathematical modelling and analysis of the COVID-19. In \cite{piovella2020analytical}, the basic SEIR epidemic model is used to study and explain some analytical results for the asymptotic and peak values and their characteristic times of the susceptible human populations affected by the highly contagious COVID-19 disease. A SLIAR-type epidemic model is used to study COVID-19 infections in China \cite{arino2020simple}. Estimated basic reproduction numbers for the COVID-19 infectious disease transmission dynamics in Italy and China have been carried out in \cite{wangping2020extended}, using a modified classical SIR mathematical model characterized by time-dependent transmission rates. A  prediction and data-driven based SEIRQ COVID-19 nonlinear infection model is formulated and studied in \cite{cui2020dynamic}. The authors in \cite{hu2020dynamic} have developed and analyzed a nonlinear epidemic model to explain the spreading dynamics of the 2019 coronavirus among the susceptible human population, the environment as well as wild animals. Two novel data-driven compartmental models are proposed in \cite{mushayabasa2020role, garba2020modeling} to investigate the COVID-19 pandemic in South Africa.

Mathematical modelling tools are essential in studying infectious diseases epidemiology because they can at least give some insight into the spreading dynamics of disease outbreaks and help in suggesting possible control strategies. The authors in \cite{ fatima2020modeling} have constructed and analyzed a non-autonomous differential equation model by introducing medical mask, isolation, treatment, and detergent spray as time-dependent controls. Global parameter sensitivity analysis for a new COVID-19 differential equation model is carried out in the work of Ali and co-authors \cite{ali2020role}. They also proposed and analyzed a non-autonomous epidemic model for the COVID-19 disease in the same work using quarantine and isolation as time-dependent control functions.
 Furthermore, a COVID-19 mathematical is studied in recent work by the authors in \cite{lemecha2020optimal}, where they considered three time-dependent control functions consisting of preventive control measures (quarantine, isolation, social distancing), disinfection of contaminated surfaces to reduce intensive medical care and infected individuals in the population. A non-optimal and optimal control deterministic  COVID-19 models are studied in \cite{aldila2020optimal}. The authors explored control and preventive interventions such as rapid testing,  medical masks, improvement of medical treatment in hospitals, and community awareness. An optimal control nonlinear epidemic model for COVID-19 infection that captures optimal preventive and control strategies such as personal protection measures, treatment of hospitalized individuals, and public health education is formulated and analyzed to study the dynamics of the epidemic in Ethiopia \cite{deressa2020modeling}. Optimal Control analysis for the 2019 coronavirus epidemic has been studied using non-pharmaceutical control and preventive interventions to examine the dynamics of the disease in the USA \cite{perkins2020optimal}. The work in \cite{oud2021fractional} studied a fractional-order mathematical model for COVID-19 dynamics with quarantine, isolation, and environmental viral load. Asamoah et.al.\cite{asamoah2020global} presented a COVID-19 model to study the impact of the environment on the spread of the disease in Ghana. They further investigated the economic outcomes using cost-effectiveness analysis. Alqarni \cite{alqarni2020mathematical} formulated and analyzed a novel deterministic COVID-19 epidemic model characterized by nonlinear differential equations with six state variables to describe the COVID-19 dynamics in the Kingdom of Saudi Arabia.  They gave a detailed qualitative stability analysis and also determined the influential model parameters on the basic reproduction number, $\mathcal{R}_0$, using global sensitivity analysis. They further performed numerical simulations to support their theoretical results, following their novel mathematical modelling formulation, analysis, and the generated global sensitivity analysis results. We are motivated to present a cost-effectiveness analysis for the work in \cite{alqarni2020mathematical}. In recent times, cost-effectiveness analysis of epidemic optimal control models has become very important in suggesting realistic optimal control strategies to help reduce the spread of infectious diseases in limited-resource settings. Also, assessing the amount it cost to acquire a unit of a health outcome like infection averted, susceptibility prevented, life-year gained, or death prevented, and the expenses and well-being results of at least one or more interventions. See, e.g., \cite{seidu2020optimal,omame2020analysis,asamoah2021sensitivity,okyere2020analysis, panja2019optimal, agusto2014optimal, agusto2019optimal, berhe2020optimal, momoh2018optimal, olaniyi2020modelling, tilahun2017modelling, berhe2019co, berhe2018optimal,abidemi2020optimal,asamoah2021non,seidu2021mathematical,khan2021robust} and some of the references therein. There are vast area of research on COVID-19 where author(s) did not either study optimal version of their proposed model or that of the economic impact of their model see for example (\cite{abidemi2021impact,iddrisu2021predictive,acheampong2021modelling,otoo2021estimating,asamoah2020mathematical,khan2020dynamics,khan2020modeling,moore2021global,ngonghala2020mathematical,sun2020transmission,eikenberry2020mask,lopez2021modified,abay2021mathematical,bassey2021global} e.t.c), therefore, we hope that this article encourages researchers to dive deeper in investigating the economic-health impacts of COVID-19 models without optimal control analysis and cost-effectiveness analysis as done in \cite{agusto2013optimal}.
 The rest of the paper is organised as follows: Section \ref{ssec1} presents the general description of the model states, and transition terms from Alqarni \cite{alqarni2020mathematical}, Section \ref{ssec2} gives the bases for the formulation of the optimal control model,  the proof of existence and the characterisation of the optimal control problem. Section \ref{ssec3} contains the numerical simulations for the various control strategies and cost-effectiveness analysis. Section \ref{ssec:6} contain the concluding remarks.
 \section{The autonomous model}\label{ssec1}
The formulated model is divided into five distinct human compartments, identified as, the susceptible, $S(t)$,
exposed, $E(t)$, asymptomatic infected (not showing symptoms but infected other healthy people) $A(t)$,
symptomatic infected (that have symptoms of disease and infect other people) $I(t)$, and the recovered individuals, $R(t)$, where the total population is given as $N(t) = S(t)+E(t)+A(t)+I(t)+R(t)$.
The assumed concentration of the SARS-CoV-2 in the environment is denoted by $B(t)$. Individuals in the infected classes $E(t), A(t), I(t)$ are assumed of transmitting the disease to the susceptible individuals at the rate $\beta_1, \beta_2, \beta_3$, respectively, and $\beta_4$ is the propensity rate of susceptible individuals getting the virus through the environment. The set of differential equations for the autonomous system is given as
\begin{align}\label{SEIR} \frac{dS}{dt}&=\Lambda-\big(\beta_{1}E+\beta_{2}I+\beta_{3}A\big)\frac{S}{N}-\beta_{4}B\frac{S}{N}-dS, \notag \\
\frac{dE}{dt}&=\big(\beta_{1}E+\beta_{2}I+\beta_{3}A\big)\frac{S}{N}+\beta_{4}B\frac{S}{N}-(\delta+d)E,\notag\\
	\frac{dI}{dt}&=(1-\tau)\delta E-(d+d_1+\gamma_{1})I,\\
	\frac{dA}{dt}&=\tau \delta E-(d+\gamma_{2})A,   \notag\\
	\frac{dR}{dt}&=\gamma_{1}I+ \gamma_{2}A-dR,\notag\\
	\frac{dB}{dt}&=\psi_{1}E +\psi_{2}I+\psi_{3}A-\phi B, \notag
\end{align}
with the initial conditions \[S(0) = S_0 > 0, E(0) = E_0 \geq 0, I(0) = I_0 \geq 0, A(0) = A_0 \geq 0, R(0) = R_0 \geq 0, B(0) = B_0 \geq 0.\]
The model's recruitment rate is given as $\Lambda$ with $d$ representing the natural death rate. The Greek symbols $\beta_1, \beta_2,\beta_3,$ are the respective direct transmission rates among exposed and susceptible individuals, infected (showing symptoms) and susceptible individuals, symptomatically infected (not showing symptoms) and susceptible individuals, and $\beta_4$ is the indirect transmission of the virus to the susceptible individuals. The rate at which the exposed individuals develops symptoms become infected is denoted as $(1-\tau)\delta$, where the rate of new asymptomatic infection is represented as $\tau\delta$. The disease-induced death rate is denoted as $d_1$. Here, the symptomatic, asymptomatic recovery rate is epidemiological assumed as $\gamma_1$  and $\gamma_2$, respectively. Furthermore, the epidemiological rates for shedding the virus into the environment by the exposed, infected and asymptomatically infected people is denoted as $\psi_1, \psi_2$ and $\psi_3$ respectively. The rate of natural removal of the virus from the environment is denoted as $\phi$. Alqarni et.al\cite{alqarni2020mathematical} gave the basic reproduction expression, detailed qualitative stability analysis and also determined the influential model parameters on the basic reproduction number, $\mathcal{R}_0$, using global sensitivity analysis. They further performed numerical simulations to support their theoretical results. The basic reproduction number from Alqarni et.al\cite{alqarni2020mathematical} is given as
\begin{equation}\label{k1}
  \mathcal{R}_0= \frac{k_2(\delta\tau(\beta_4\psi_3 + \beta_3\phi)+k_3(\beta_4\psi_1 +\beta_1\phi))+\delta k_3(1-\tau)(\beta_4\psi_2 + \beta_2 \phi)}{k_1k_2k_3\phi},
\end{equation}
where $k_1 = d+\delta, k_2 =\gamma_1 +d +d_1$ and $k_3 =\gamma_2+d$.
Following their global sensitivity result of the basic reproduction number, section \ref{ssec2} is conceived.
 \section{Optimal control problem formulation and analysis of COVID-19 model}\label{ssec2}
In section 4.1 of the work in Alqarni et al. \cite{alqarni2020mathematical}.\, They found out that the most sensitive parameters in their basic reproduction number are: Contact rate among exposed and susceptible, $\beta_1$, contact rate among environment and susceptible, $\beta_4$, virus contribution due to state $E$ to compartment $B$, $\psi_1$, and virus removal from the environment, $\theta$. Therefore, to contribute the research knowledge on COVID-19 in Saudi Arabia, we incorporated the following control terms to study the most effective economic and health outcomes in combating this disease which has caused economic hardship in many countries.
\subsection{Formulation of the non-autonomous COVID-19 model}
\begin{itemize}
	\item $u_1$: practising physical or social distancing protocols.
	\item $u_2$:  practising personal hygiene by cleaning contaminated surfaces with alcohol based detergents.
	\item $u_3$: practising proper and safety measures by exposed, asymptomatic infected and symptomatic infected individuals.
	\item $u_4$: fumigating schools in all levels of education, sports facilities, commercial areas and religious worship centres.
\end{itemize}
Hence, based on \cite{alqarni2020mathematical}, our optimal control model of  is given as
\begin{align}
	\label{optimalmodelSEIR}
	\frac{dS}{dt}&=\Lambda-(1-u_1(t))\big(\beta_{1}E+\beta_{2}I+\beta_{3}A\big)\frac{S}{N}-(1-u_1(t)-u_2(t))\beta_{4}B\frac{S}{N}-dS, \notag \\
	\frac{dE}{dt}&=(1-u_1(t))\big(\beta_{1}E+\beta_{2}I+\beta_{3}A\big)\frac{S}{N}+(1-u_1(t)-u_2(t))\beta_{4}B\frac{S}{N}-(\delta+d)E,\notag\\
	\frac{dI}{dt}&=(1-\tau)\delta E-(d+d_1+\gamma_{1})I,\\
	\frac{dA}{dt}&=\tau \delta E-(d+\gamma_{2})A,   \notag\\
	\frac{dR}{dt}&=\gamma_{1}I+ \gamma_{2}A-dR,\notag\\
	\frac{dB}{dt}&=(1-u_3(t))\psi_{1}E +(1-u_3(t))\psi_{2}I+(1-u_3(t))\psi_{3}A-(u_4(t)+\phi) B, \notag
\end{align}
\subsection{Formulation of the objective functional}
In line with the standard in literature \cite{asamoah2020global,agusto2019optimal,berhe2020optimal,abidemi2020optimal,abidemi2019global,asamoah2017modelling,olaniyi2020mathematical,asamoah2020deterministic,asamoah2020backward,asamoah2018mathematical}, the control cost is measures by implementing a quadratic performance index or objective functional in this work. Thus, our goal is to minimize the objective functional, $\mathcal J$, given as
\begin{equation}
	\label{objectivefunctional}
	\mathcal{J}(u_1, u_2, u_3, u_4 ):=\min \int_{0}^{T}\bigg[A_{1}E+A_{2}I+A_3 A+A_4 B+ \dfrac{1}{2}\sum_{i=1}^{4}D_{i}u^{2}_{i}(t)\bigg]dt
\end{equation}
subject to the non-autonomous system \eqref{optimalmodelSEIR}, where $A_i>0$ ($i=1,2,3,4$) are the balancing weight constants on the exposed, asymptomatic and symptomatic infected individuals, and the concentration of corona virus in the environment respectively, whereas $D_i>0$ are the balancing cost factors on the respective controls $u_i$ (for $i=1,\ldots,4$), $T$ is the final time for controls implementation.

Suppose $\mathcal U$ is a non-empty control set defined by
\begin{equation}
\label{u}
\mathcal U=\left\{(u_1,u_2,u_3,u_4):u_i \ \mbox{Lebesgue measurable}, \ 0\le u_i\le 1, \mbox{for $i=1,\ldots,4$}, \ t\in[0,T]\right\}.
\end{equation}
Then, it is of particular interest to seek an optimal control quadruple $u^*=(u_1^*,u_2^*,u_3^*,u_4^*)$ such that
\begin{equation}
\label{optimalj}
\mathcal J(u^*)=\min\left\{\mathcal J(u_1,u_2,u_3,u_4):u_1,u_2,u_3,u_4\in\mathcal U\right\}.
\end{equation}

\subsection{Existence of an optimal control}

\begin{theorem}
\label{theorem-exist}
Given the objective functional $\mathcal J$ defined on the control set
$\mathcal U$ in \eqref{u}, then there exists an optimal control quadruple $u^*=(u_1^*,u_2^*,u_3^*,u_4^*)$ such that \eqref{optimalj} holds when the following conditions are satisfied \cite{okyere2020analysis,berhe2018optimal,olaniyi2020modelling}:
\begin{enumerate}
\item[(i)] The admissible control set $\mathcal U$ is convex and closed.
\item[(ii)] The state system is bounded by a linear function in the state and control variables.
\item[(iii)] The integrand of the objective functional $\mathcal J$ in \eqref{objectivefunctional} is convex in respect of the controls.
\item[(iv)] The Lagrangian is bounded below by
\begin{equation*}
a_0\left(\sum_{i=1}^4|u_i|^2\right)^\frac{a_2}{2}-a_1,
\end{equation*}
where $a_0,a_1>0$ and $a_2>1$.
\end{enumerate}
\end{theorem}

\begin{proof}
Let the control set $\mathcal U=[0,u_{\max}]^4, u_{\max}\leq 1$ , $u=(u_1,u_2,u_3,u_4)\in\mathcal U$, $x=(S,E,I,A,R,B)$ and $f_0(t,x,u)$ be the right-hand side of the non-autonomous system \eqref{optimalmodelSEIR} given by
\begin{equation}
\label{rhs}
f_0(t,x,u)=\left(
\begin{array}{c}
\Lambda-(1-u_1)\frac{(\beta_1E+\beta_2I+\beta_3A)S}{N}-(1-u_1-u_2)\frac{\beta_4BS}{N}-dS\\
(1-u_1)\frac{(\beta_1E+\beta_2I+\beta_3A)S}{N}+(1-u_1-u_2)\frac{\beta_4BS}{N}-(\delta+d)E\\
(1-\tau)\delta E-(d+d_1+\gamma_1)I\\
\tau\delta E-(d+\gamma_2)A\\
\gamma_1I+\gamma_2A-dR\\
(1-u_3)\psi_1E+(1-u_3)\psi_2I+(1-u_3)\psi_3A-(u_4+\phi)B
\end{array}
\right).
\end{equation}
Then, we proceed by verifying the four properties presented by Theorem \ref{theorem-exist}.
\begin{enumerate}
\item [(i)] Given the control set $\mathcal U=[0,u_{max}]^4$. Then, by definition, $\mathcal U$ is closed. Further, let $v,w\in\mathcal U$, where $v=(v_1,v_2,v_3,v_4)$ and $w=(w_1,w_2,w_3,w_4)$, be any two arbitrary points. It then follows from the definition of a convex set \cite{rector2005principles}, that
\begin{equation*}
\lambda v+(1-\lambda)w\in[0,u_{max}]^4, \ \mbox{for all} \ \lambda\in[0,u_{max}].
\end{equation*}
Consequently, $\lambda v+(1-\lambda)w\in\mathcal U$, implying the convexity of $\mathcal U$.

\item[(ii)] This property is verified by adopting the ideas of the previous authors \cite{olaniyi2020modelling,romero2018optimal}. Obviously, $f_0(t,x,u)$ in \eqref{rhs} can be written as
\begin{equation*}
f_0(t,x,u)=f_1(t,x)+f_2(t,x)u,
\end{equation*}
where
\begin{equation*}
f_1(t,x)=\left(
\begin{array}{c}
\Lambda-\frac{(\beta_1E+\beta_2I+\beta_3A)S}{N}-\frac{\beta_4BS}{N}-dS\\
\frac{(\beta_1E+\beta_2I+\beta_3A)S}{N}+\frac{\beta_4BS}{N}-(\delta+d)E\\
(1-\tau)\delta E-(d+d_1+\gamma_1)I\\
\tau\delta E-(d+\gamma_2)A\\
\gamma_1I+\gamma_2A-dR\\
\psi_1E+\psi_2I+\psi_3A-\phi B
\end{array}
\right)
\end{equation*}
and
\begin{equation*}
f_2(t,x)=\left(
\begin{array}{cccc}
\frac{(\beta_1E+\beta_2I+\beta_3A)S}{N}+\frac{\beta_4BS}{N} & \frac{\beta_4BS}{N} & 0 & 0\\
-\frac{(\beta_1E+\beta_2I+\beta_3A)S}{N}-\frac{\beta_4BS}{N} & -\frac{\beta_4BS}{N} & 0 & 0\\
0 & 0 & 0 & 0\\
0 & 0 & 0 & 0\\
0 & 0 & 0 & 0\\
0 & 0 & -(\psi_1E+\psi_2I+\psi_3A) & -B
\end{array}
\right).
\end{equation*}
Hence,
\begin{align*}
\parallel f_0(t,x,u)\parallel&\le\parallel f_1(t,x)\parallel+f_2(t,x)\parallel\parallel u\parallel\\
&\le b_1+b_2\parallel u\parallel,
\end{align*}
where $b_1$ and $b_2$ are positive constants given as
\begin{equation*}
b_1=\sqrt{\max\left\{c_1,c_2\right\}\left(\Lambda^2+\Lambda^2(\psi_1+\psi_2+\psi_3)^2\right)},
\end{equation*}
and
\begin{equation*}
b_2=\sqrt{\max\left\{d_1,d_2\right\}\left(\Lambda^2+\Lambda^2(\psi_1+\psi_2+\psi_3)^2\right)},
\end{equation*}
with
\begin{align*}
c_0&=d^2+\beta_3(2\beta_1+2\beta_2+\beta_3)+\left(\beta_1+\beta_2\right)^2+(1-2\tau+2\tau^2)\delta^2+\left(\gamma_1+\gamma_2\right)^2+\psi_3(2\psi_1+2\psi_2+\psi_3)\\
&\quad+(\psi_1+\psi_2)^2,\\
c_1&=\frac{c_0}{d^2}, \\
c_2&=\frac{2\beta_4(\beta_1\phi+\beta_2\phi+\beta_3\phi)+\beta_4^2(\psi_1+\psi_2+\psi_3)}{d^2\phi^2(\psi_1+\psi_2+\psi_3)},\\
d_1&=\frac{2\beta_3(2\beta_1+2\beta_2+\beta_3)+2(\beta_1+\beta_2)^2+(\psi_1+\psi_3)^2+\psi_2(2\psi_1+\psi_2+2\psi_3)}{d^2},\\
d_2&=\frac{4\beta_4\phi(\beta_1+\beta_2+\beta_3)+(1+4\beta_4^2)(\psi_1+\psi_2+\psi_3)}{d^2\phi^2(\psi_1+\psi_2+\psi_3)}.
\end{align*}

\item[(iii)] First note that the objective functional $\mathcal J$ in \eqref{optimalj} has an integrand of the Lagrangian form defined as
\begin{equation}
\label{lang}
\mathcal L(t,x,u)=A_1E+A_2I+A_3A+A_4B=\frac{1}{2}\sum_{i=1}^4D_iu_i^2.
\end{equation}
Let $v=(v_1,v_2,v_3,v_4)\in\mathcal U$, $w=(w_1.w_2,w_3,w_4)\in\mathcal U$ and $\lambda\in[0,u_{max}]$, then it suffices to prove that
\begin{equation}
\label{ineq}
\mathcal L(t,x,(1-\lambda)v+\lambda w)\le(1-\lambda)\mathcal L(t,x,v)+\lambda\mathcal L(t,x,w).
\end{equation}
From \eqref{lang},
\begin{equation}
\label{lang1}
\mathcal L(t,x,(1-\lambda)v+\lambda w)=A_1E+A_2I+A_3A+A_4B+\frac{1}{2}\sum_{i=1}^4\left[D_i((1-\lambda)v_i+\lambda w_i)^2\right],
\end{equation}
and
\begin{equation}
\label{lang2}
(1-\lambda)\mathcal L(t,x,v)+\lambda\mathcal L(t,x,w)=A_1E+A_2I+A_3A+A_4B+\frac{1}{2}(1-\lambda)\sum_{i=1}^4D_iv_i^2+\frac{1}{2}\lambda\sum_{i=1}^4D_iw_i^2.
\end{equation}
Applying the inequality \eqref{ineq} to the results in \eqref{lang1} and \eqref{lang2} leads to
\begin{equation*}
\mathcal L(t,x,(1-\lambda)v+\lambda w)-((1-\lambda)\mathcal L(t,x,v)+\lambda\mathcal L(t,x,w))=\frac{1}{2}(\lambda^2-\lambda)\sum_{i=1}^4D_i(v_i-w_i)^2\le0, \ \mbox{since} \ \lambda\in[0,u_{max}],
\end{equation*}
implying that the integrand $\mathcal L(t,x,u)$ of the objective functional $\mathcal J$ is convex.
\item [(iv)] Lastly, the fourth property is verified as follows:
\begin{align*}
\mathcal L(t,x,u)&=A_1E+A_2I+A_3A+A_4B=\frac{1}{2}\sum_{i=1}^4D_iu_i^2,\\
&\ge\frac{1}{2}\sum_{i=1}^4D_iu_i^2,\\
&\ge a_0\left(|u_1|^2+|u_2|^2+|u_3|^2+|u_4|^2\right)^\frac{a_2}{2}-a_1,
\end{align*}
where $a_0=\frac{1}{2}\max\left\{D_1,D_2,D_3,D_4\right\}$, $a_1>0$ and $a_2=2$.
\end{enumerate}
\end{proof}

\subsection{Characterization of the optimal controls}

Pontryagin's maximum principle (PMP) provides the necessary conditions that an optimal control quadruple must satisfy. This principle converts the optimal control problem consisting of the non-autonomous system \eqref{optimalmodelSEIR} and the objective functional $\mathcal J$ in \eqref{objectivefunctional} into an issue of minimizing pointwise a
Hamiltonian, denoted as $\mathcal H$, with respect to controls $u=(u_1,u_2,u_3,u_4)$. First, the Hamiltonian $\mathcal H$ associated with the optimal control problem is formulated as
\begin{align}
	\label{eqnhamiltonianSIRD}
	\mathcal{H}&=A_{1}E+A_{2}I+A_3 A+A_4 B+ \dfrac{1}{2}\bigg(D_{1}u^{2}_{1}+D_{2}u^{2}_{2}+D_{3}u^{2}_{3}+D_{4}u^{2}_{4}\bigg)\notag\\
	&+\lambda_{1}\bigg[\Lambda-(1-u_1)\big(\beta_{1}E+\beta_{2}I+\beta_{3}A\big)\frac{S}{N}-(1-u_1-u_2)\beta_{4}B\frac{S}{N}-dS\bigg]\notag\\
	&+\lambda_{2}\bigg[(1-u_1)\big(\beta_{1}E+\beta_{2}I+\beta_{3}A\big)\frac{S}{N}+(1-u_1-u_2)\beta_{4}B\frac{S}{N}-(\delta+d)E\bigg]\notag\\
	&+\lambda_{3}\big[(1-\tau)\delta E-(d+d_1+\gamma_{1})I\big]\\
	&+\lambda_{4}\big[\tau \delta E-(d+\gamma_{2})A \big]\notag\\
	&+\lambda_{5}\big[\gamma_{1}I+ \gamma_{2}A-dR\big]\notag\\
	&+\lambda_{6}\big[(1-u_3)\psi_{1}E +(1-u_3)\psi_{2}I+(1-u_3)\psi_{3}A-(u_4+\phi) B\big]\notag,
\end{align}
where $\lambda_i$ (with $i=1,2,\ldots,6$) are the adjoint variables corresponding to the state variables $S$, $E$, $I$, $A$, $R$ and $B$ respectively.

\begin{theorem}
If $u^*=(u_i^*), \ i=1,\ldots,4$ is an optimal control quadruple and $S^*$, $E^*$, $I^*$, $A^*$, $R^*$, $B^*$ are the solutions of the corresponding state system \eqref{optimalmodelSEIR} that minimizes $\mathcal J(u^*)$ over the control set $\mathcal U$ defined by \eqref{u}, then there exist adjoint variables $\lambda_i \ (i=1,2,\ldots,6)$ satisfying
\begin{align}\label{adjointequation}
\dfrac{d\lambda_{1}}{dt}&=(\lambda_{1}-\lambda_{2})\bigg[(1-u_1)\big( \beta_{1}{E^*}+\beta_{2}{I^*}+\beta_{3}{A^*}\big)\bigg]\frac{({E^*}+{I^*}+{A^*}+{R^*})}{N^2}\notag\\ &+(\lambda_{1}-\lambda_{2})(1-u_1-u_2)\beta_{4}{B^*}\frac{({E^*}+{I^*}+{A^*}+{R^*})}{N^2}+\lambda_{1}d,\notag\\
\dfrac{d\lambda_{2}}{dt}&=-A_1+(\lambda_{1}-\lambda_{2})(1-u_1){S^*}\bigg[\dfrac{({S^*}+{I^*}+{A^*}+{R^*})\beta_{1}-(\beta_{2}{I^*}+\beta_{3}{A^*})}{N^2}\bigg]\notag\\
&+(\lambda_{2}-\lambda_{1})(1-u_1-u_2)\beta_{4}{B^*}\frac{{S^*}}{N^2}+(\delta+d)\lambda_{2}-\lambda_{2}(1-\tau)\delta\lambda_{3}-\tau\delta\lambda_{4}-(1-u_3)\psi_{1}\lambda_{6},\\
\dfrac{d\lambda_{3}}{dt}&=-A_2+(\lambda_{1}-\lambda_{2})(1-u_1){S^*}\bigg[\dfrac{({S^*}+{E^*}+{A^*}+{R^*})\beta_{2}-(\beta_{1}{E^*}+\beta_{3}{A^*})}{N^2}\bigg]\notag\\
&+(\lambda_{2}-\lambda_{1})(1-u_1-u_2)\beta_{4}{B^*}\frac{{S^*}}{N^2}+(\delta+d_1+\gamma_{1})\lambda_{3}-\gamma_{1}\lambda_{5}-(1-u_3)\psi_{2}\lambda_{6},\notag\\
\dfrac{d\lambda_{4}}{dt}&=-A_3+(\lambda_{1}-\lambda_{2})(1-u_1){S^*}\bigg[\dfrac{({S^*}+{E^*}+{I^*}+{R^*})\beta_{3}-(\beta_{1}{E^*}+\beta_{2}{I^*})}{N^2}\bigg]\notag\\
&+(\lambda_{2}-\lambda_{1})(1-u_1-u_2)\beta_{4}{B^*}\frac{{S^*}}{N^2}+(d_1+\gamma_{2})\lambda_{4}-\gamma_{2}\lambda_{5}-(1-u_3)\psi_{3}\lambda_{6},\notag\\
\dfrac{d\lambda_{5}}{dt}&=(\lambda_{2}-\lambda_{1})(1-u_1)\bigg(\beta_{1}{E^*}+\beta_{2}{I^*}+\beta_{3}{A^*}\bigg)\frac{{S^*}}{N^2} +(\lambda_{2}-\lambda_{1})(1-u_1-u_2)\beta_{4}\frac{{S^*}}{N^2}+\lambda_{5}d, \notag\\
\dfrac{d\lambda_{6}}{dt}&=-A_4+(\lambda_{1}-\lambda_{2})(1-u_1-u_2)\beta_{4}\frac{{S^*}}{N} +(u_4+\phi)\lambda_{6}d .\notag
\end{align}
with transversality conditions

\begin{equation}
\label{trans}
\lambda_i(T)=0, \quad \ \mbox{for} \ i=1,2,\ldots,6
\end{equation}

and
\begin{equation}\label{OPTIMALcontFUNCTIONS}
\begin{cases}
u^{*}_{1} (t)=\min\Bigg\{\max\Bigg\{0, \dfrac{(\lambda_{2}-\lambda_{1})(\beta_{1}{E^*}+\beta_{2}I^*+\beta_{3}A^*+\beta_{4}B^*)S}{D_{1}N}\Bigg\},u_{1_{max}} \Bigg\},\\
u^{*}_{2} (t)=\min\Bigg\{\max\Bigg\{0, \dfrac{(\lambda_{2}-\lambda_{1})\beta_{4}{B^*}S^*}{D_{2}N}\Bigg\}, u_{2_{max}}\Bigg\},\\
u^{*}_{3} (t)=\min\Bigg\{\max\Bigg\{0, \dfrac{\lambda_6 (\psi_{1}E^*+\psi_{2}I^*+\psi_{3}A^*)}{D_{3}}\Bigg\},u_{3_{max}} \Bigg\},\\
u^{*}_{4} (t)=\min\Bigg\{\max\Bigg\{0, \dfrac{\phi {B^*}\lambda_6}{D_4}\Bigg\},u_{4_{\max}}\Bigg\}.
\end{cases}
\end{equation}
\end{theorem}

\begin{proof}
The form of the adjoint system and the transversality conditions associated with this optimal control problem follows the widely used standard results obtained from work done by Pontryagin et al. \cite{pontryagin1962mathematical}. For this purpose, we partially differentiate the formulated Hamiltonian function~(\ref{eqnhamiltonianSIRD}) with respect to $S, E, I, A, R$ and $B$ as follows;

\begin{equation}\label{COVIDy2k}
\begin{cases}
\frac{d\lambda_{1}}{dt}=-\frac{\partial \mathcal{H}}{\partial S}, \qquad \lambda_1(T)=0,\\
\frac{d\lambda_{2}}{dt}=-\frac{\partial \mathcal{H}}{\partial E},  \qquad \lambda_2(T)=0,\\
\frac{d\lambda_{3}}{dt}=-\frac{\partial \mathcal{H}}{\partial I},  \qquad \lambda_3(T)=0,\\
\frac{d\lambda_{4}}{dt}=-\frac{\partial \mathcal{H}}{\partial A},  \qquad \lambda_4(T)=0,\\
\frac{d\lambda_{3}}{dt}=-\frac{\partial \mathcal{H}}{\partial R}, \qquad \lambda_5(T)=0, \\
\frac{d\lambda_{4}}{dt}=-\frac{\partial \mathcal{H}}{\partial B}, \qquad \lambda_6(T)=0.
\end{cases}
\end{equation}

Finally, to obtain the desired results for the characterizations of the optimal
control, we need to partially differentiate the Hamiltonian function~(\ref{eqnhamiltonianSIRD}) with respect to the four time-dependent control functions ($u_1, u_2, u_3, u_4$), thus, further, the optimal control characterization in \eqref{OPTIMALcontFUNCTIONS} is obtained by solving
\begin{equation*}
\frac{\partial\mathcal H}{\partial u_i}=0, \qquad \mbox{for} \ u_i^* \qquad \mbox{(where $i=1,2,3,4$)}.
\end{equation*}

Lastly, it follows from standard control arguments involving bounds on the control that
\begin{equation*}
u_i^*=\begin{cases}
0 &\quad \mbox{if} \quad \theta_i^*\le0,\\
\theta_i^* &\quad \mbox{if} \quad 0\le\theta_i^*\le u_{\max},\\
 u_{\max} &\quad \mbox{if} \quad \theta_i^*\ge  u_{\max},
\end{cases}
\end{equation*}	
where $i=1,2,3,4$ and with
\begin{align*}
	\theta_1&=\frac{(\lambda_{2}-\lambda_{1})(\beta_{1}E+\beta_{2}I+\beta_{3}A+\beta_{4}B)S}{D_{1}N},\\
	\theta_2&=\frac{(\lambda_{2}-\lambda_{1})\beta_{4}BS}{D_{2}N},\\
	\theta_3&=\frac{\lambda_6 (\psi_{1}E+\psi_{2}I+\psi_{3}A)}{D_{3}},\\
	\theta_4&=\frac{\phi B\lambda_6}{D_4}.
\end{align*}

\end{proof}

\section{Numerical simulation and cost-effectiveness analysis}\label{ssec3}

\subsection{Numerical simulation}
\label{sec-numericsim}
Numerical simulations are vital in dynamical modelling; they give the proposed model's pictorial view to the theoretical analysis. Hence, we provide the numerical outcomes of our study by simulating 14 possible strategic combinations of the control measures. This is done by simulating the constraint system \eqref{optimalmodelSEIR} froward in time and the adjoint system \eqref{eqnhamiltonianSIRD} backward in time until convergence is reached. The model parameters can be found in \cite{alqarni2020mathematical}, but restated here for easy reference see Table \ref{TW}.
 This simulation procedure is popularly known as fourth-order Runge-Kutta forward-backward sweep simulations. The 14 possible strategic combination strategies are divided into four scenarios, thus, the implementation of single control (Scenario A), the use of dual controls (Scenario B), the performance of triple controls (Scenario C) and lastly, the implementation of quadruplet control measures (Scenario D). Iterated below as
\begin{itemize}
	\item[$\spadesuit$] Scenario A (implementation of single control)
	\begin{itemize}
		\item[$\maltese$] Strategy 1: practising physical or social distancing protocols only\\
		$(u_1\neq 0, u_2=u_3=u_4=0).$
		\item[$\maltese$] Strategy 2: practising personal hygiene by cleaning contaminated surfaces with alcohol based detergents only ($u_1=0, u_2\neq0, u_3=u_4=0).$
		\item[$\maltese$] Strategy 3: practising proper and safety measures by exposed, asymptomatic infected and symptomatic infected individuals only \\
		$(u_1=0, u_2=0, u_3\neq 0, u_4=0).$
		\item[$\maltese$] Strategy 4: Fumigating schools in all levels of education, sports facilities, commercial areas and religious worship centres only $(u_1= 0, u_2=0,u_3=0,u_4\neq0).$
	\end{itemize}
\end{itemize}
\begin{itemize}
	\item[$\spadesuit$] Scenario B (the use of double controls)
	\begin{itemize}
		\item[$\maltese$] Strategy 5: practising physical or social distancing protocols + practising personal hygiene by cleaning contaminated surfaces with alcohol based detergents  \\
		$(u_1, u_2\neq 0, u_3=u_4=0).$
		\item[$\maltese$] Strategy 6: practising physical or social distancing protocols +  practising proper and safety measures by exposed, asymptomatic infected and symptomatic infected individuals ($u_1, u_3\neq0, u_2=u_4=0).$
		\item[$\maltese$] Strategy 7: practising physical or social distancing protocols + fumigating schools in all levels of education, sports facilities, commercial areas and religious worship centres
		$(u_1, u_4\neq 0, u_2=0, u_3=0).$
		\item[$\maltese$] Strategy 8: practising personal hygiene by cleaning contaminated surfaces with alcohol based detergents  +  practising proper and safety measures by exposed, asymptomatic infected and symptomatic infected individuals
		$(u_2, u_3\neq0 ,u_1=0,u_4=0).$
		\item[$\maltese$] Strategy 9: practising personal hygiene by cleaning contaminated surfaces with alcohol based detergents  +  fumigating schools in all levels of education, sports facilities, commercial areas and religious worship centres
		$(u_2, u_4\neq0 ,u_1=0,u_3=0).$
		\item[$\maltese$] Strategy 10: practising proper and safety measures by exposed, asymptomatic infected and symptomatic infected individuals  +  fumigating schools in all levels of education, sports facilities, commercial areas and religious worship centres
		$(u_3, u_4\neq0 ,u_1=0,u_2=0).$
	\end{itemize}
\end{itemize}
\begin{itemize}
	\item[$\spadesuit$] Scenario C (the use of triple controls)
	\begin{itemize}
		\item[$\maltese$] Strategy 11: practising physical or social distancing protocols + practising personal hygiene by cleaning contaminated surfaces with alcohol based detergents + practising proper and safety measures by exposed, asymptomatic infected and symptomatic infected individuals
		$(u_1, u_2,u_3\neq 0, u_4=0).$
		\item[$\maltese$] Strategy 12: practising physical or social distancing protocols +  practising personal hygiene by cleaning contaminated surfaces with alcohol based detergents +  fumigating schools in all levels of education, sports facilities, commercial areas and religious worship centres ($u_1, u_2, u_4\neq0, u_3=0).$
		\item[$\maltese$] Strategy 13: practising personal hygiene by cleaning contaminated surfaces with alcohol based detergents + practising proper and safety measures by exposed, asymptomatic infected and symptomatic infected individuals + fumigating schools in all levels of education, sports facilities, commercial areas and religious worship centres
		$(u_2, u_3, u_4\neq 0, u_1=0).$
	\end{itemize}
\end{itemize}
\begin{itemize}
	\item[$\spadesuit$] Scenario D (implementation of quadruplet)
	\begin{itemize}
		\item[$\maltese$] Strategy 14: practising physical or social distancing protocols + practising personal hygiene by cleaning contaminated surfaces with alcohol based detergents + practising proper and safety measures by exposed, asymptomatic infected and symptomatic infected individuals +  fumigating schools in all levels of education, sports facilities, commercial areas and religious worship centres
		$(u_1, u_2,u_3, u_4\neq 0).$
	\end{itemize}
\end{itemize}
\begin{center}
\begin{table}[h!]\center\caption{Model's parameter descriptions and values. \label{TW}}
\resizebox{15cm}{!}{\setlength{\tabcolsep}{11pt}
\renewcommand{\arraystretch}{1} \begin{tabular}{c l c c}
 \toprule
 \bf{Parameter}& \bf{Definition}&\bf{Value} &\bf{Source}  \\
 \toprule
$\Lambda$ & Recruitment rate & $d \times N(0)$& \cite{alqarni2020mathematical}\\
$d$& Natural mortality rate& $\frac{1}{74.87\times 365}$ &\cite{alqarni2020mathematical}\\
$\beta_1$ &Contact rate among exposed and susceptible& 0.1233 &\cite{alqarni2020mathematical}\\
$\beta_2$& Contact rate among infected (symptomatic) and susceptible& 0.0542 &\cite{alqarni2020mathematical}\\
$\beta_3$ &Contact rate among infected (asymptomatic) and susceptible &0.0020& \cite{alqarni2020mathematical}\\
$\beta_4$ &Contact rate among environment and susceptible& 0.1101 &\cite{alqarni2020mathematical}\\
$\delta$ &Incubation period &0.1980 &\cite{alqarni2020mathematical}\\
$\tau$ & fraction that transient to $A$ &0.3085& \cite{alqarni2020mathematical}\\
$d_1$ &Natural death rate due to Infection at I &0.0104 &\cite{alqarni2020mathematical}\\
$\gamma_1$& Recovery from I &0.3680 &\cite{alqarni2020mathematical}\\
$\gamma_2$& Recovery from A &0.2945 &\cite{alqarni2020mathematical}\\
$\psi_1$& Virus contribution due to E to B& 0.2574 &\cite{alqarni2020mathematical}\\
$\psi_2$ &Virus contribution due to I to B & 0.2798 &\cite{alqarni2020mathematical}\\
$\psi_3$ &Virus contribution due to A to B &0.1584 &\cite{alqarni2020mathematical}\\
$\phi$ &Virus removal from environment& 0.3820& \cite{alqarni2020mathematical}\\
\bottomrule
\end{tabular}}
\end{table}
\end{center}
\subsubsection{Scenario A: use of single control}
\begin{figure}[h!]
	\centering
	\begin{subfigure}{.5\textwidth}
		\centering
		\includegraphics[width=1\linewidth, height=2in]{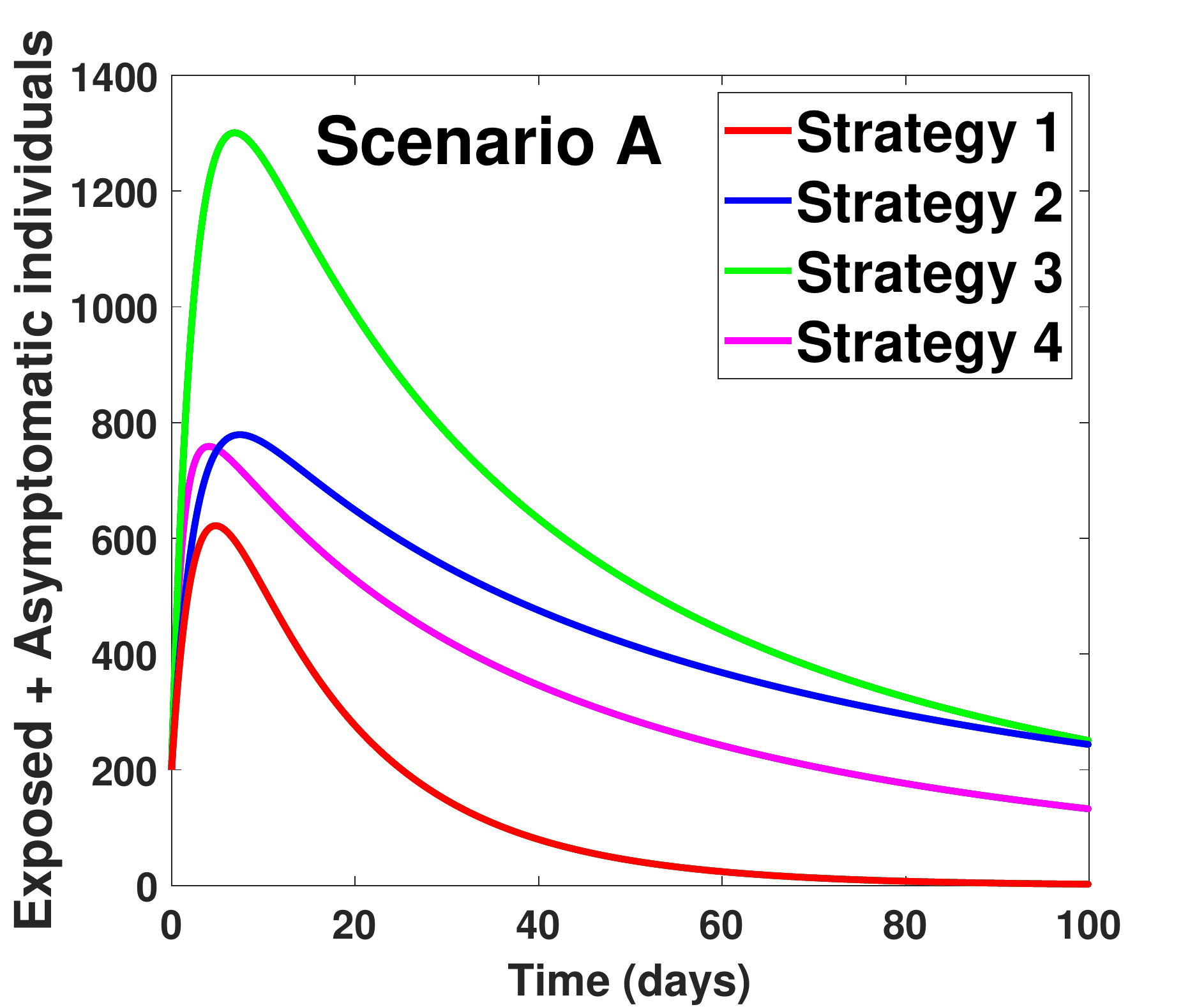}
		\caption{}\label{SA1}
	\end{subfigure}%
	\begin{subfigure}{.5\textwidth}
		\centering
		\includegraphics[width=1\linewidth, height=2in]{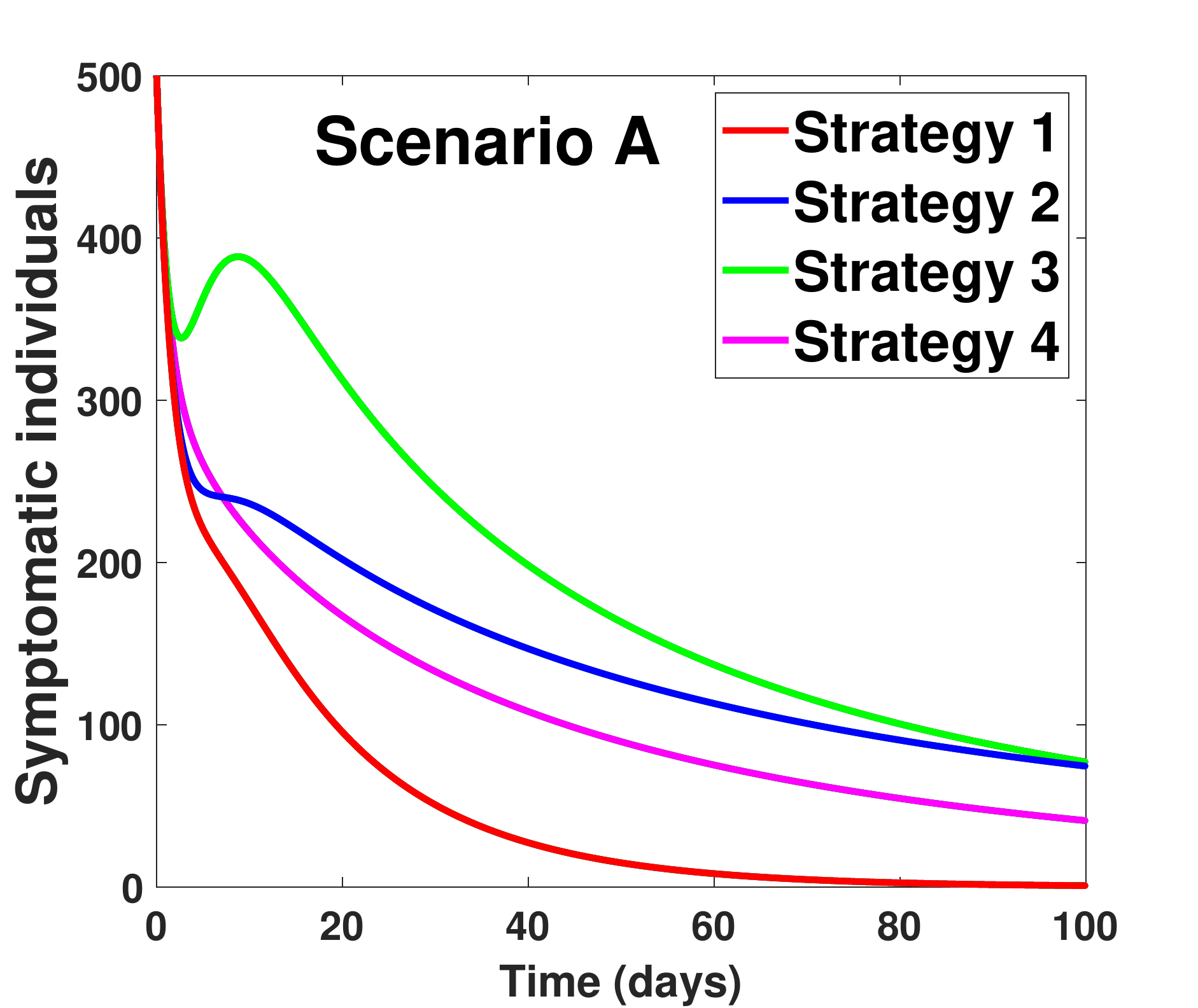}
		\caption{}\label{SA2}
	\end{subfigure}
	\begin{subfigure}{.5\textwidth}
		\centering
		\includegraphics[width=1\linewidth,  height=2in]{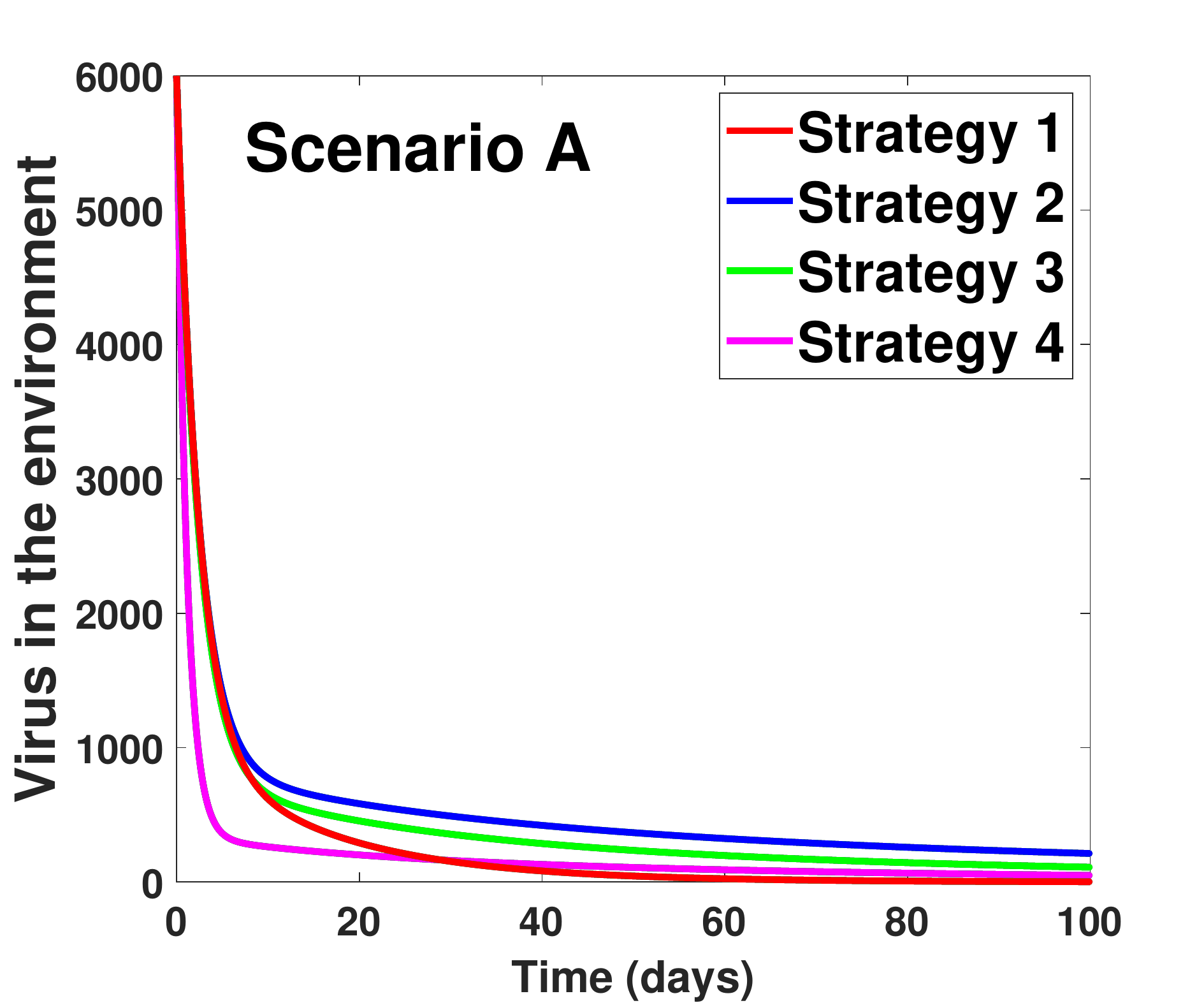}
		\caption{}\label{SA3}
	\end{subfigure}%
	\begin{subfigure}{.5\textwidth}
		\centering
		\includegraphics[width=1\linewidth, height=2in ]{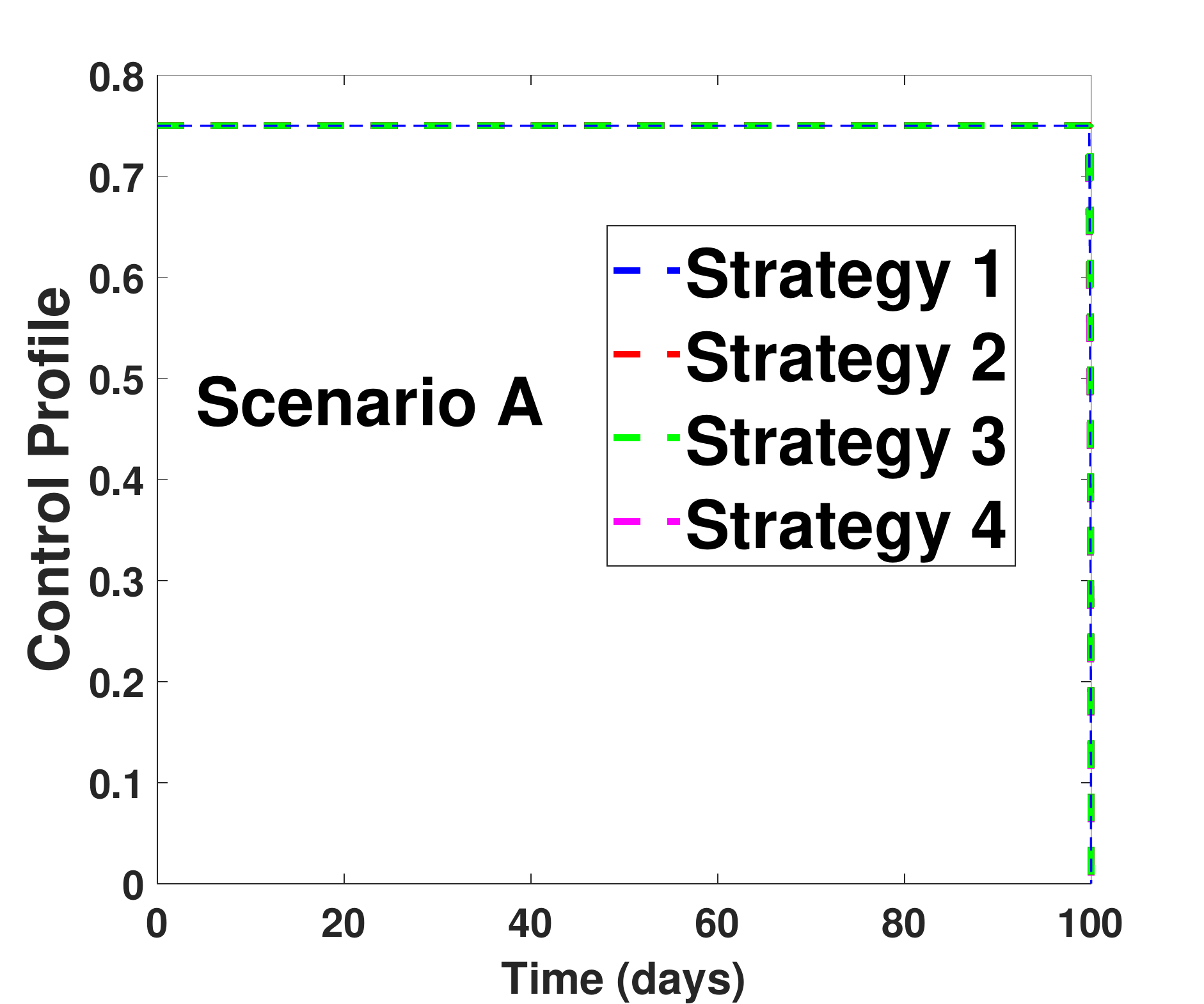}
		\caption{Control profile}\label{SA4}
	\end{subfigure}
	\begin{subfigure}{.5\textwidth}
		\centering
		\includegraphics[width=1\linewidth,  height=2in]{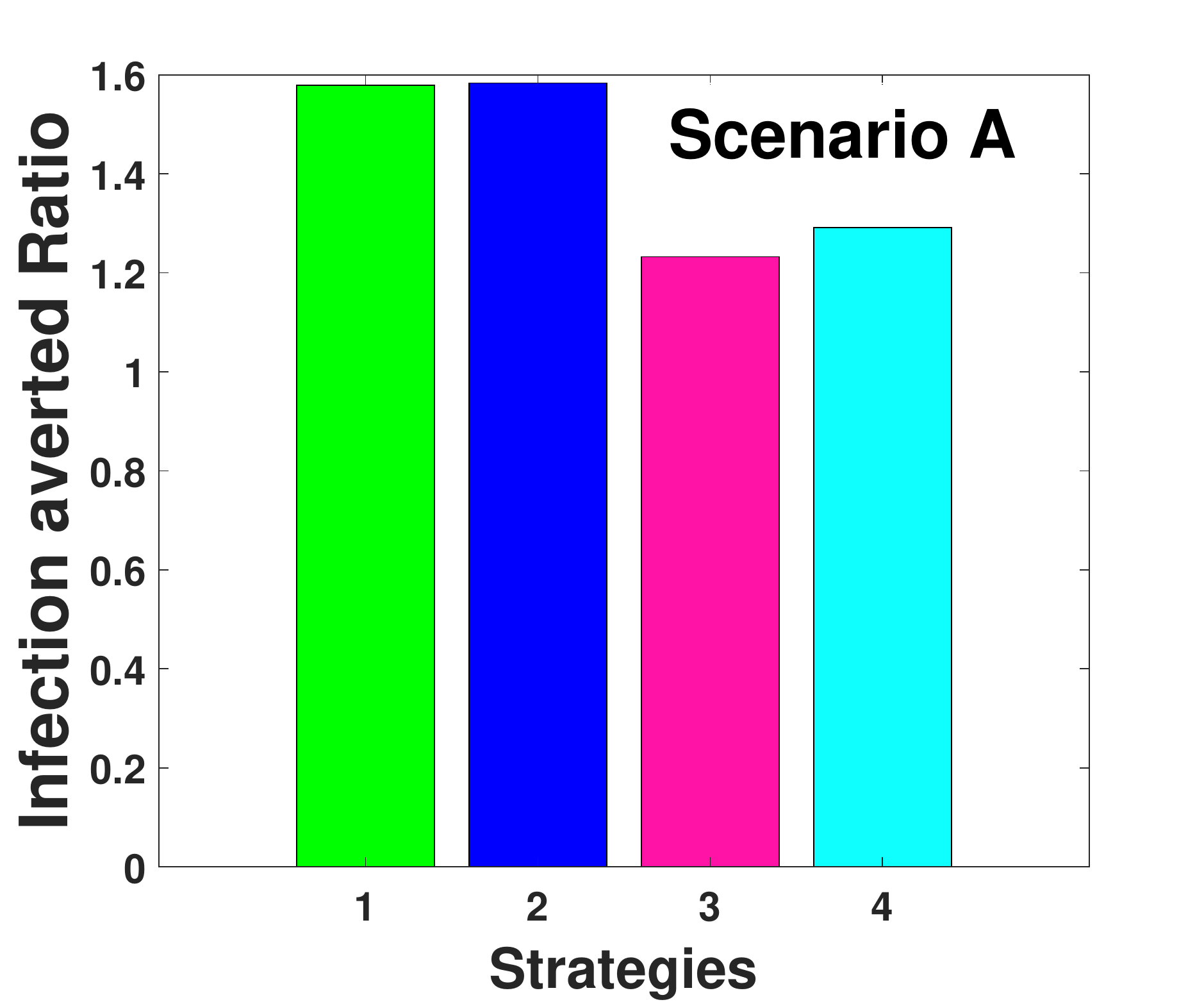}
		\caption{Infection Averted Ratio}\label{SA5}
	\end{subfigure}%
	\begin{subfigure}{.5\textwidth}
		\centering
		\includegraphics[width=1\linewidth, height=2in ]{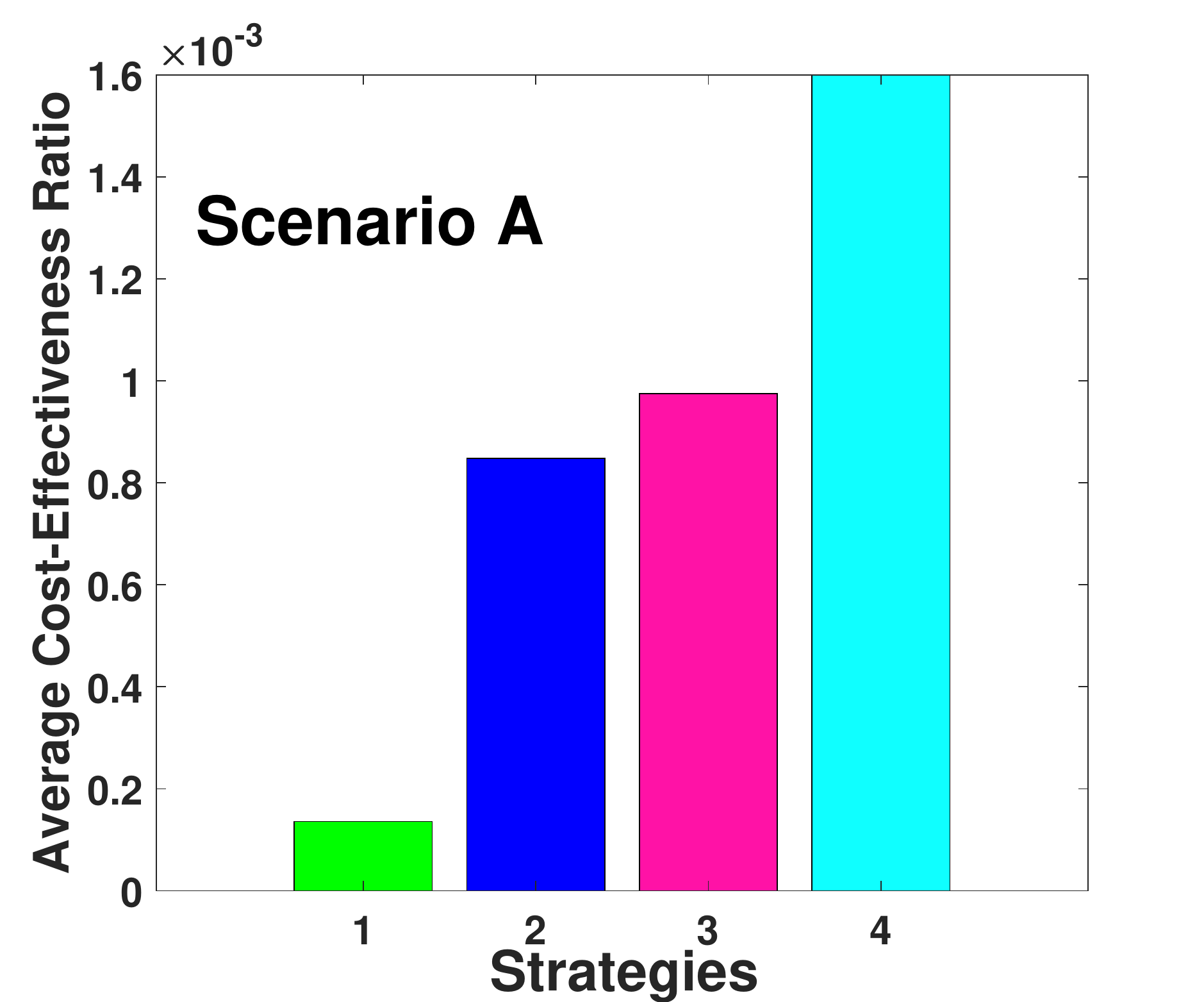}
		\caption{Average Cost-Effectiveness Ratio}\label{SA6}
	\end{subfigure}
	\caption{ Single control strategy.}\label{Singlecontrol}
\end{figure}
In Figure \ref{SA1}, we noticed that strategy 1 has the highest number of exposed and asymptomatic averted individuals, followed by strategy 4, strategy 3, and then strategy 1. Likewise, in Figure \ref{SA2} we noticed that the dynamical importance of the strategy is of equal usefulness on the number of symptomatic individuals. In Figure \ref{SA3}, we saw that the strategy with the highest number of virus removal from the environment is strategy 4, with strategy 2 having the lowest viral removal effect. In Figure \ref{SA4}, the control profiles suggest that the optimal strategy for scenario A should be implemented on the same control level. Figure \ref{SA5} shows the infection averted ratio of the various control strategies. It shows that strategy 2 (practising personal hygiene by cleaning contaminated surfaces with alcohol-based
detergents only), is the most effective strategy if health officials stick to scenario A only in controlling COVID-19 in the Kingdom of Saudi Arabia. Figure \ref{SA6} shows the average cost-effectiveness ratio, which also supports that strategy 1 is the most effective and cost-saving strategy in scenario A.  The mathematical extraction of the infection averted ratio and the average cost-effectiveness ratio can be found in subsection \ref{costeffective} where we validate the claim on Figure \ref{SA5} and \ref{SA6} respectively.
\subsubsection{Scenario B: use of double controls}
In Figure \ref{SB1}-\ref{SB4}, we carried out numerical simulations with the notion that an individual may apply two of the suggested controls simultaneously. We noticed in Figure \ref{SB1} that strategy 5 has the highest number of exposed and asymptomatic averted individuals, in the long run, followed by strategy 6, strategy 7, strategy 9, strategy 10 and then strategy 8. Likewise, in Figure \ref{SB2} we noticed that the dynamical importance of the strategy is of equal usefulness on the number of symptomatic individuals. In Figure \ref{SB3}, we saw that the strategy with the highest number of virus removal from the environment is strategy 7 and 9, with strategy 8 having the most minimal viral removal effect. In Figure \ref{SB4}, the control profiles suggest that the optimal strategy for scenario B should be implemented on a control level of 0.75 for each control term in strategy 5 and 8 for the entire simulation period. For the control strategy 9 in Figure \ref{SB4}, we noticed that the control terms in strategy 9, should be kept at 0.75 for 95 days and then reduced to 0.5 for each of the control terms for the rest of the simulation time. The control profile for strategy 10 shows that each control term should be kept for 0.75 for 92 days and then reduced to 0.5 for the rest of the simulation period. The control profile for strategy 6 shows that, with the combined effort of the two controls, the strategy control level should be kept at 1.5 for 65 days and then gradually reduced to 0.98 for the entire simulation time. We also noticed in Figure \ref{SB4} that the control profile of strategy 7 shows that the control level for the two controls in strategy 7 should be kept at 1.5, thus 0.75 each for 41 days and then gradually reduced to 1 for the entire simulation time. Figure \ref{SB5} shows the infection averted ratio of the various control strategies. It shows that strategy 5 is the most effective when one uses the infection averted ratio (IAR). Figure \ref{SB6} shows the average cost-effectiveness ratio, which indicates that strategy 6 is the most effective and cost-saving strategy in scenario B. The mathematical extraction of the infection averted ratio and the average cost-effectiveness ratio can be found in subsection \ref{costeffective} where we validate the claim on Figure \ref{SB5} and \ref{SB6} respectively.
\begin{figure}[h!]
	\centering
	\begin{subfigure}{.5\textwidth}
		\centering
		\includegraphics[width=1\linewidth, height=2in]{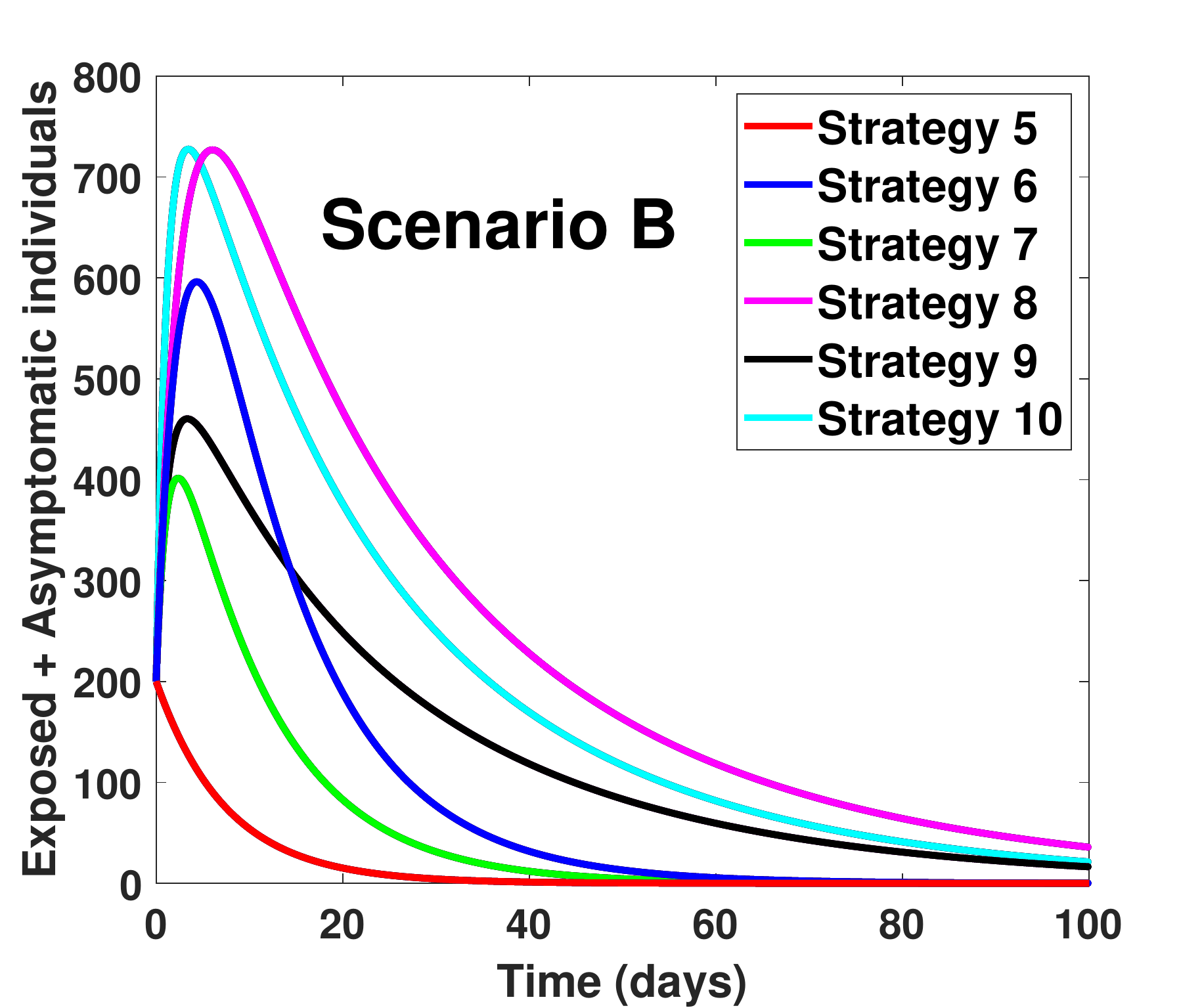}
		\caption{}\label{SB1}
	\end{subfigure}%
	\begin{subfigure}{.5\textwidth}
		\centering
		\includegraphics[width=1\linewidth, height=2in]{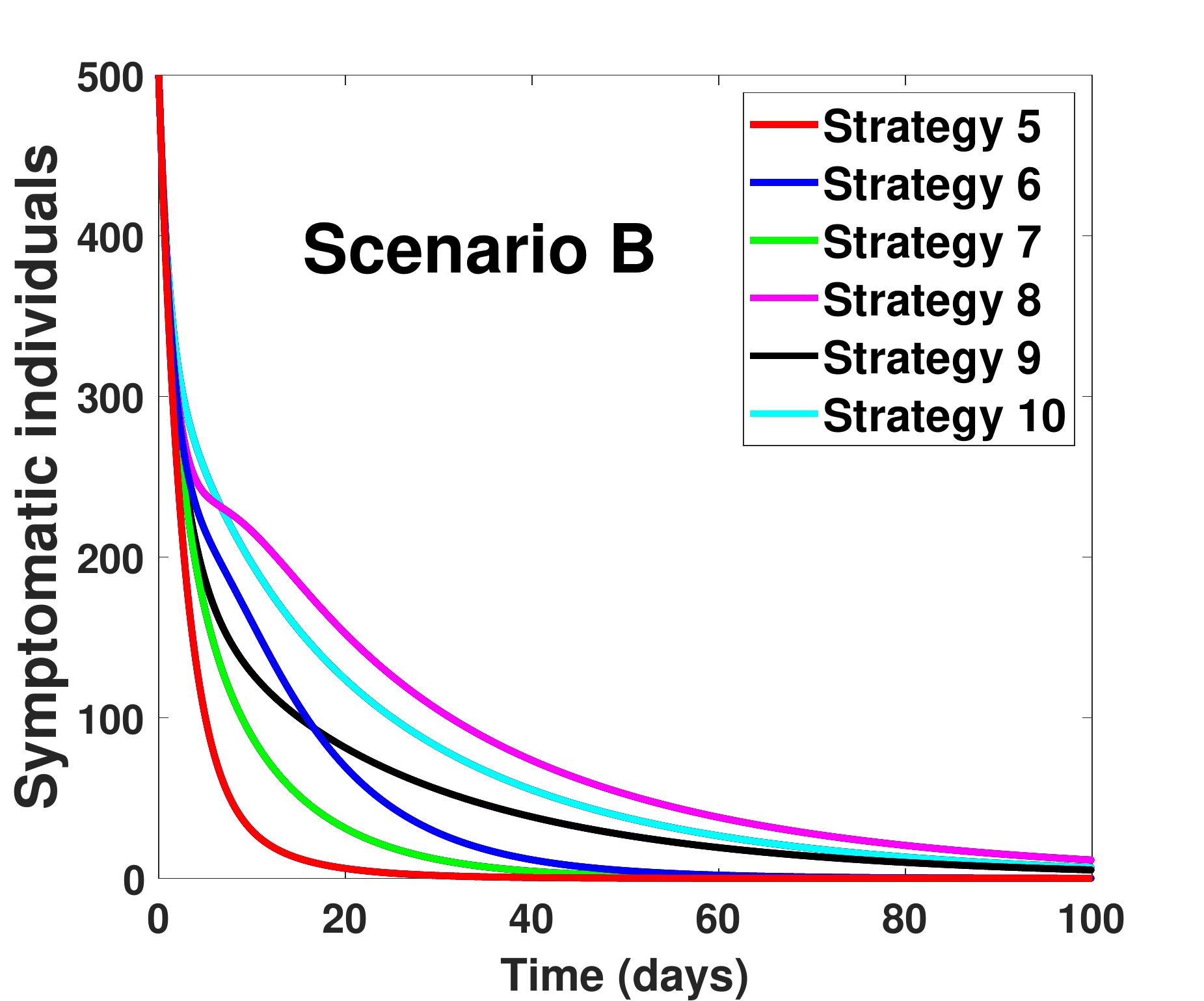}
		\caption{}\label{SB2}
	\end{subfigure}
	\begin{subfigure}{.5\textwidth}
		\centering
		\includegraphics[width=1\linewidth,  height=2in]{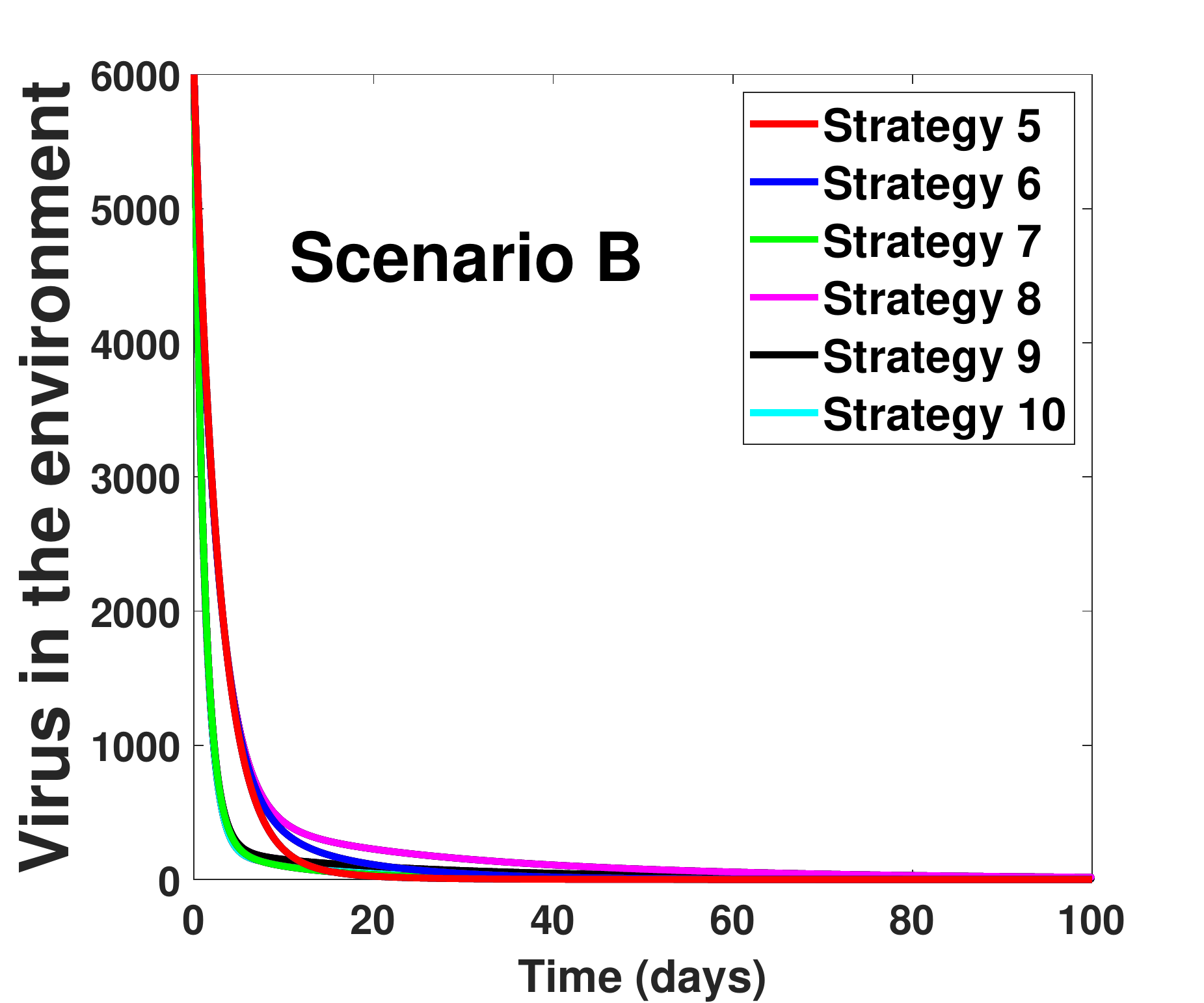}
		\caption{}\label{SB3}
	\end{subfigure}%
	\begin{subfigure}{.5\textwidth}
		\centering
		\includegraphics[width=1\linewidth, height=2in ]{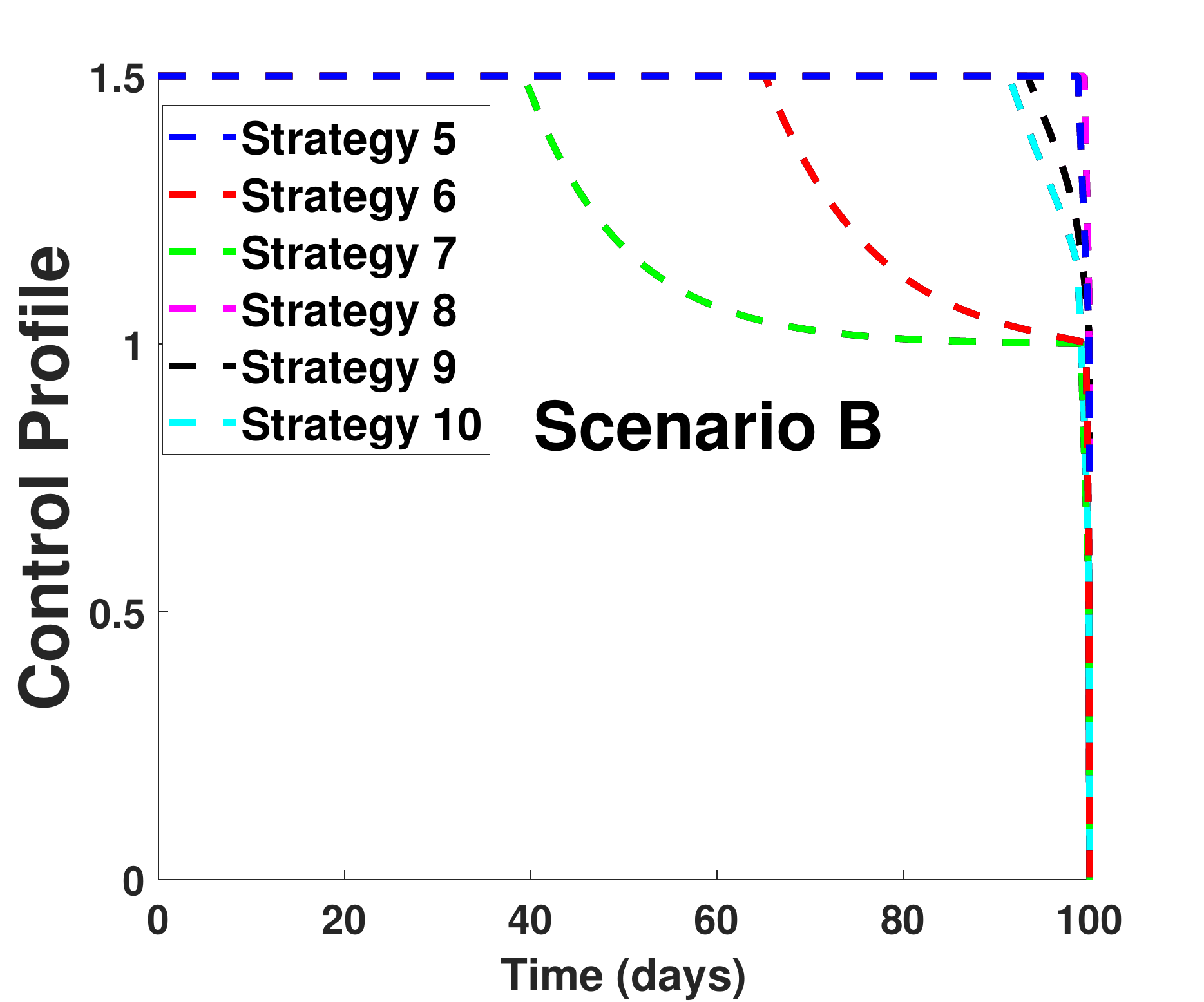}
		\caption{Control profile}\label{SB4}
	\end{subfigure}
	\begin{subfigure}{.5\textwidth}
		\centering
		\includegraphics[width=1\linewidth,  height=2in]{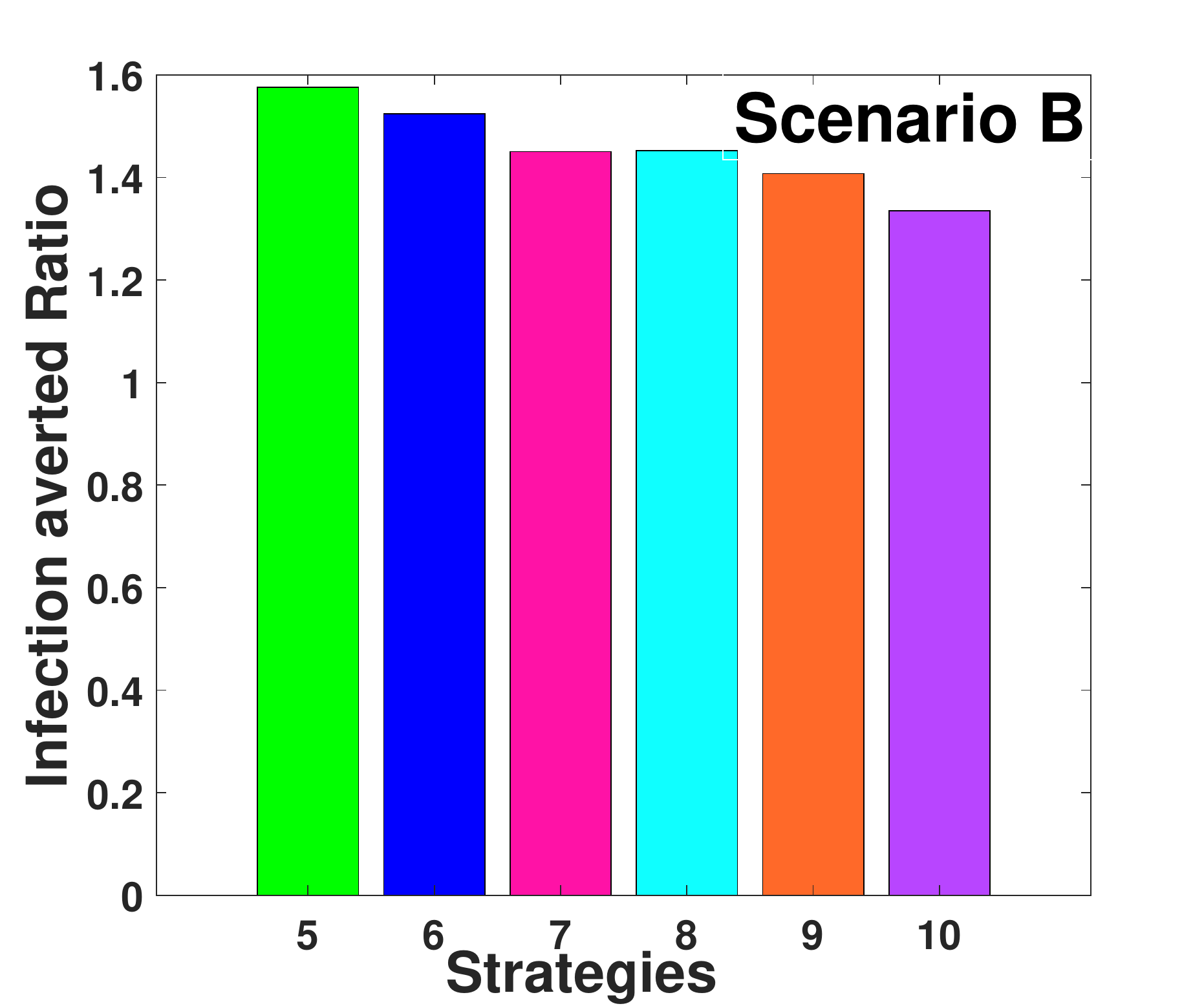}
		\caption{Infection Averted Ratio}\label{SB5}
	\end{subfigure}%
	\begin{subfigure}{.5\textwidth}
		\centering
		\includegraphics[width=1\linewidth, height=2in ]{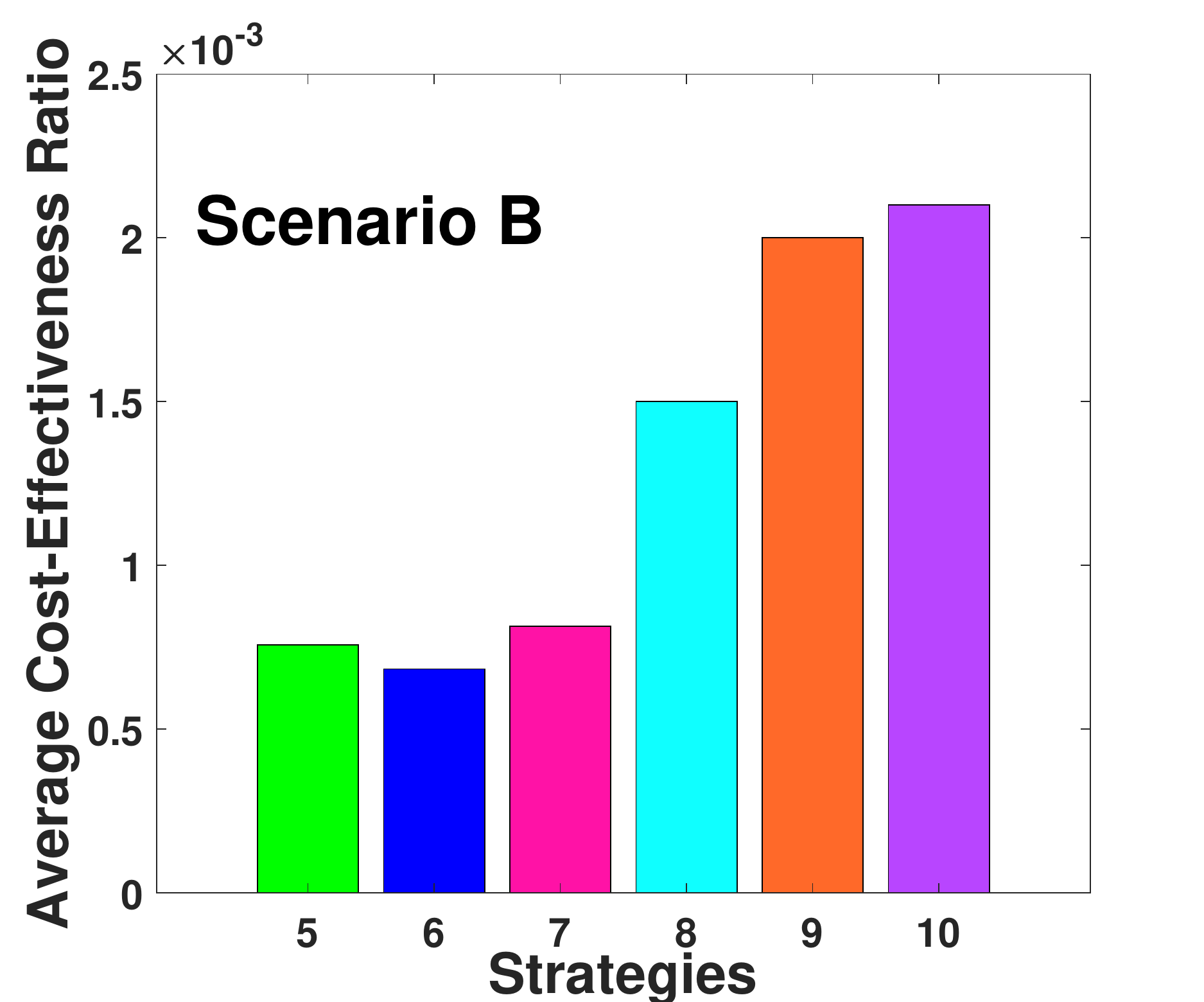}
		\caption{Average Cost-Effectiveness Ratio}\label{SB6}
	\end{subfigure}
	\caption{ Double control strategies.}\label{Doublecontrol}
\end{figure}

\subsubsection{Scenario C: use of triple controls}
In Figure \ref{SC1}-\ref{SC4}, we carried out numerical simulations with the notion that an individual may apply three of the suggested controls simultaneously. We noticed in Figure \ref{SC1} that strategy 11 has the highest number of exposed and asymptomatic averted individuals, in the long run, followed by strategy 12 and then strategy 13. Likewise, in Figure \ref{SC2} we noticed that the dynamical importance of the strategy is of equal usefulness on the number of symptomatic individuals.
\begin{figure}[h!]
	\centering
	\begin{subfigure}{.5\textwidth}
		\centering
		\includegraphics[width=1\linewidth, height=2in]{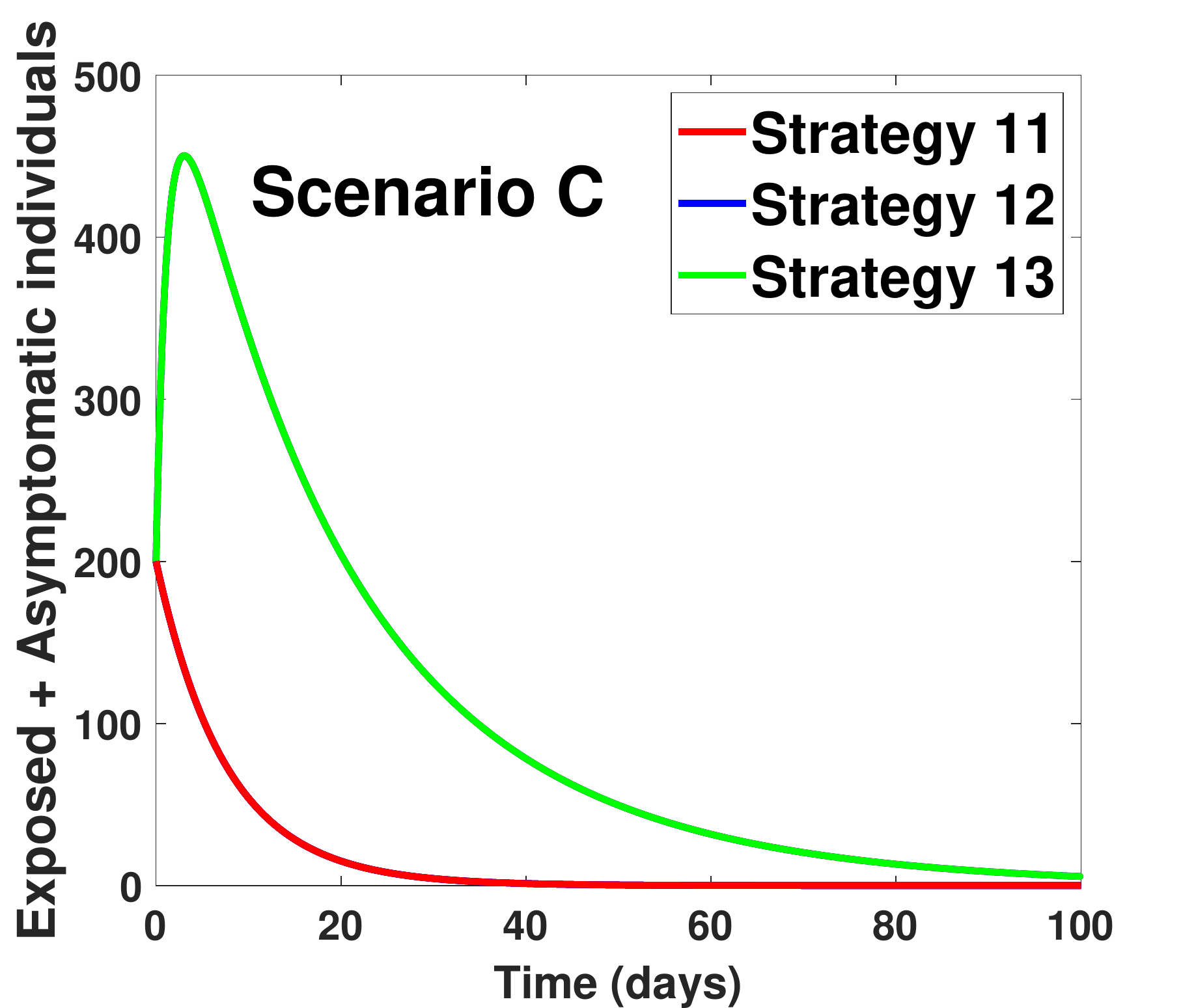}
		\caption{}\label{SC1}
	\end{subfigure}%
	\begin{subfigure}{.5\textwidth}
		\centering
		\includegraphics[width=1\linewidth, height=2in]{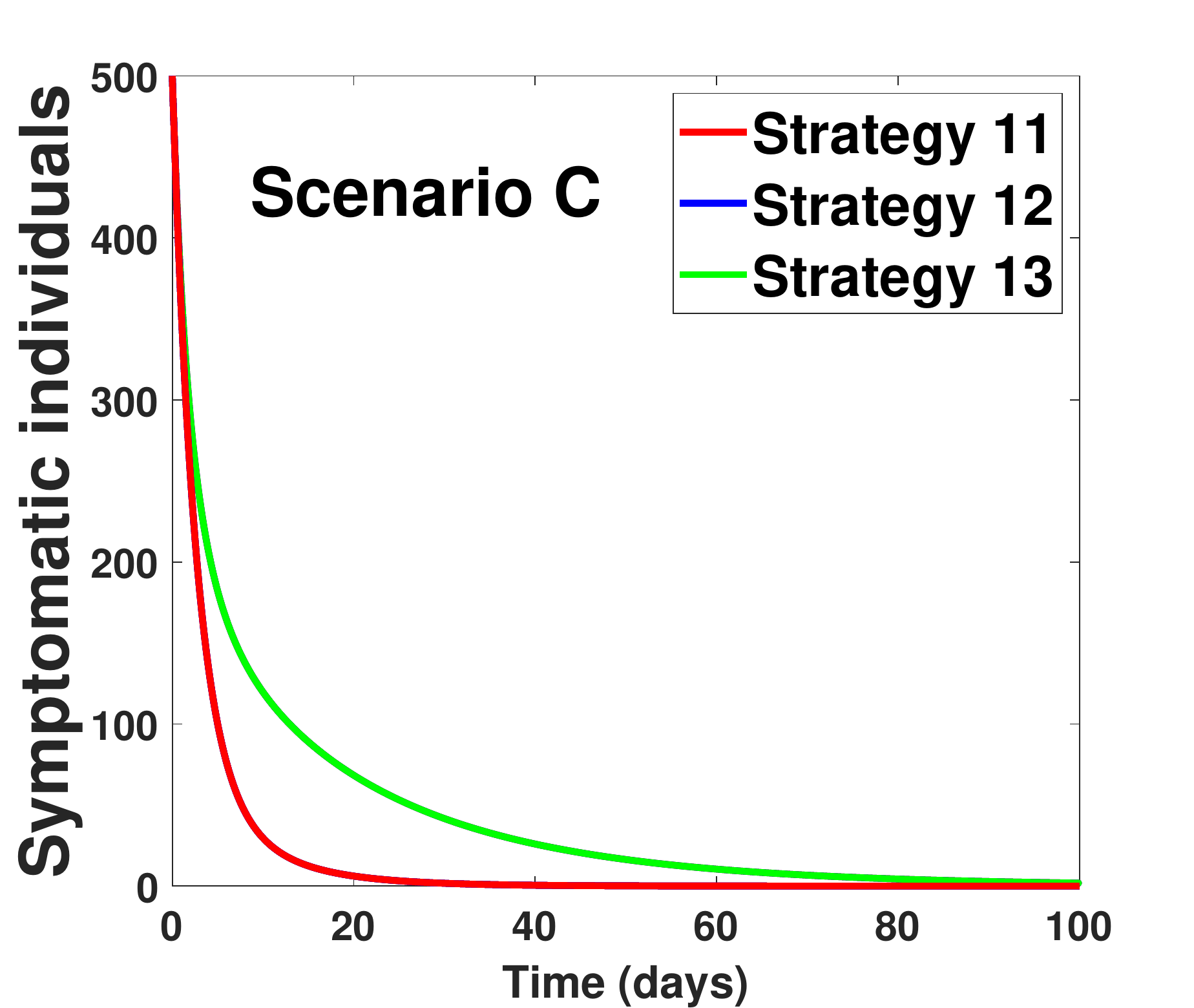}
		\caption{}\label{SC2}
	\end{subfigure}
	\begin{subfigure}{.5\textwidth}
		\centering
		\includegraphics[width=1\linewidth,  height=2in]{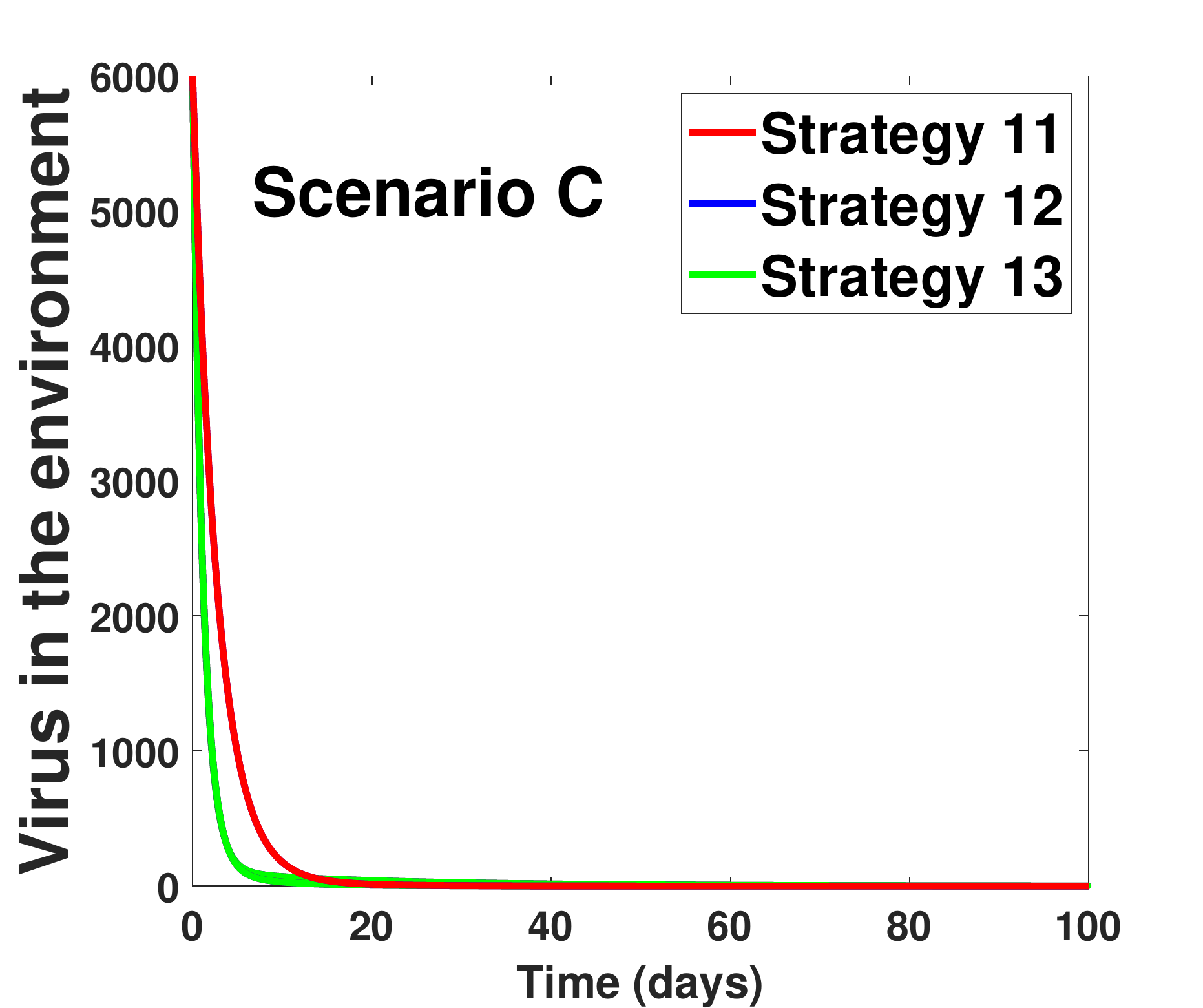}
		\caption{}\label{SC3}
	\end{subfigure}%
	\begin{subfigure}{.5\textwidth}
		\centering
		\includegraphics[width=1\linewidth, height=2in ]{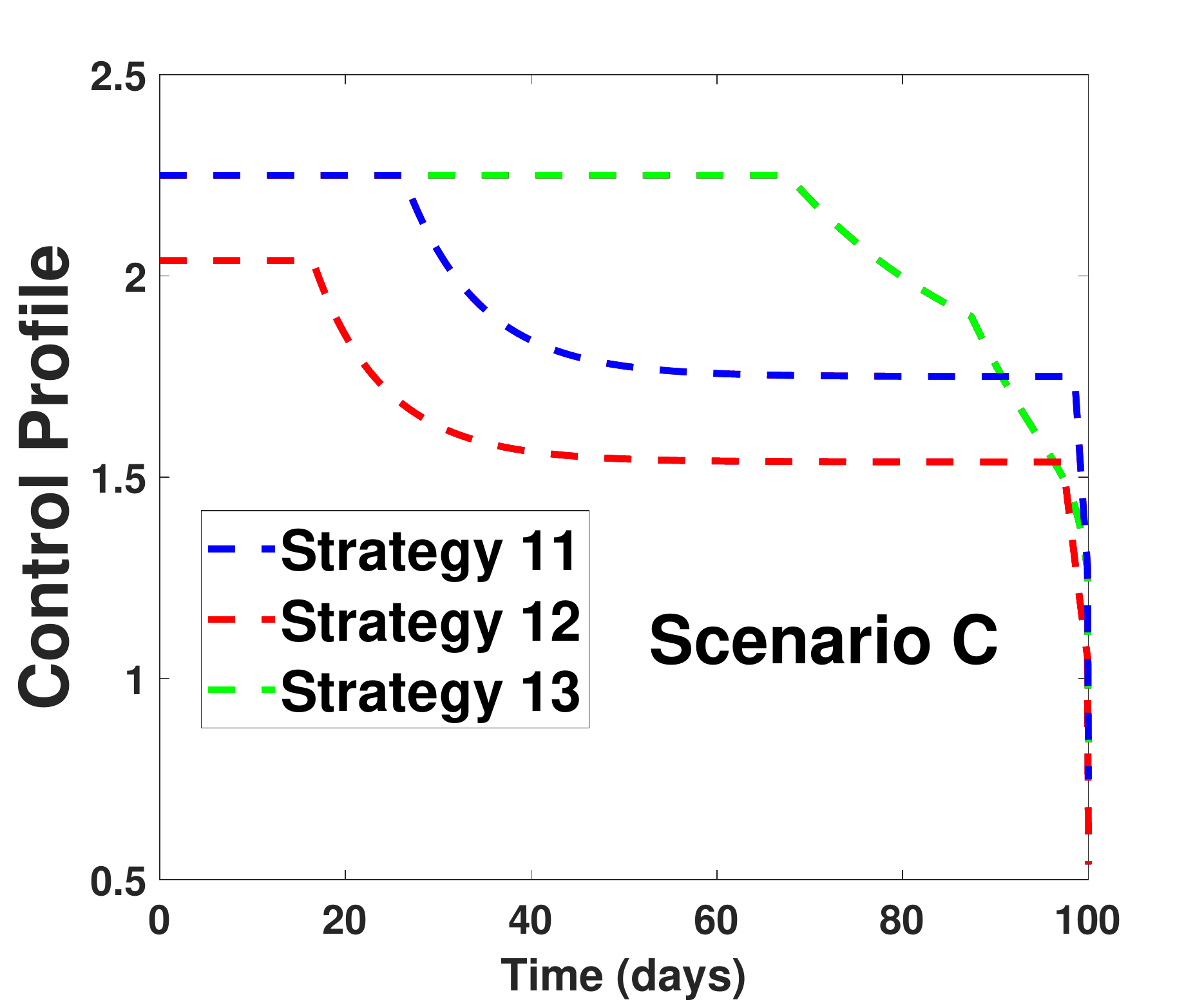}
		\caption{Control profile}\label{SC4}
	\end{subfigure}
	\begin{subfigure}{.5\textwidth}
		\centering
		\includegraphics[width=1\linewidth,  height=2in]{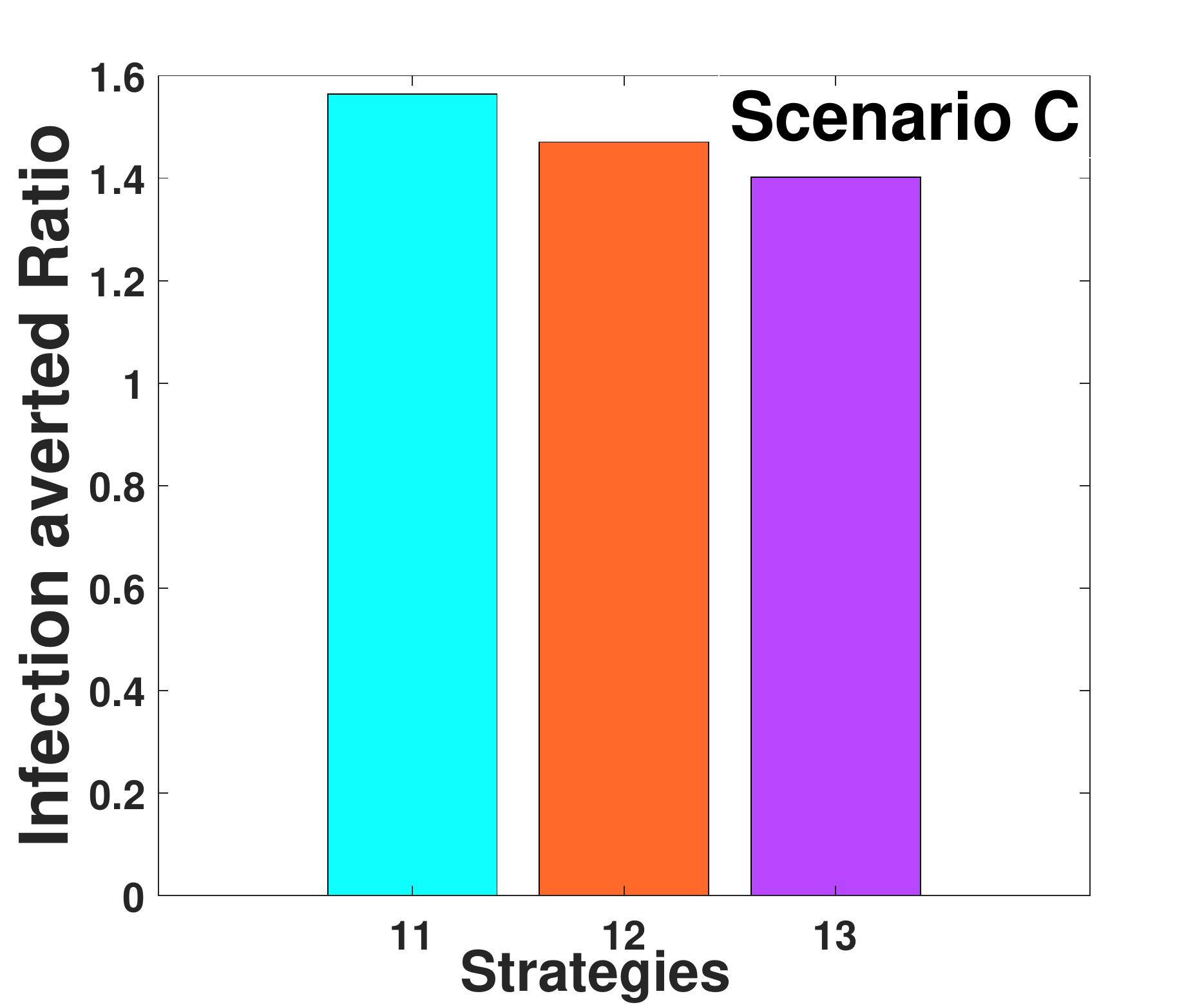}
		\caption{Infection Averted Ratio}\label{SC5}
	\end{subfigure}%
	\begin{subfigure}{.5\textwidth}
		\centering
		\includegraphics[width=1\linewidth, height=2in ]{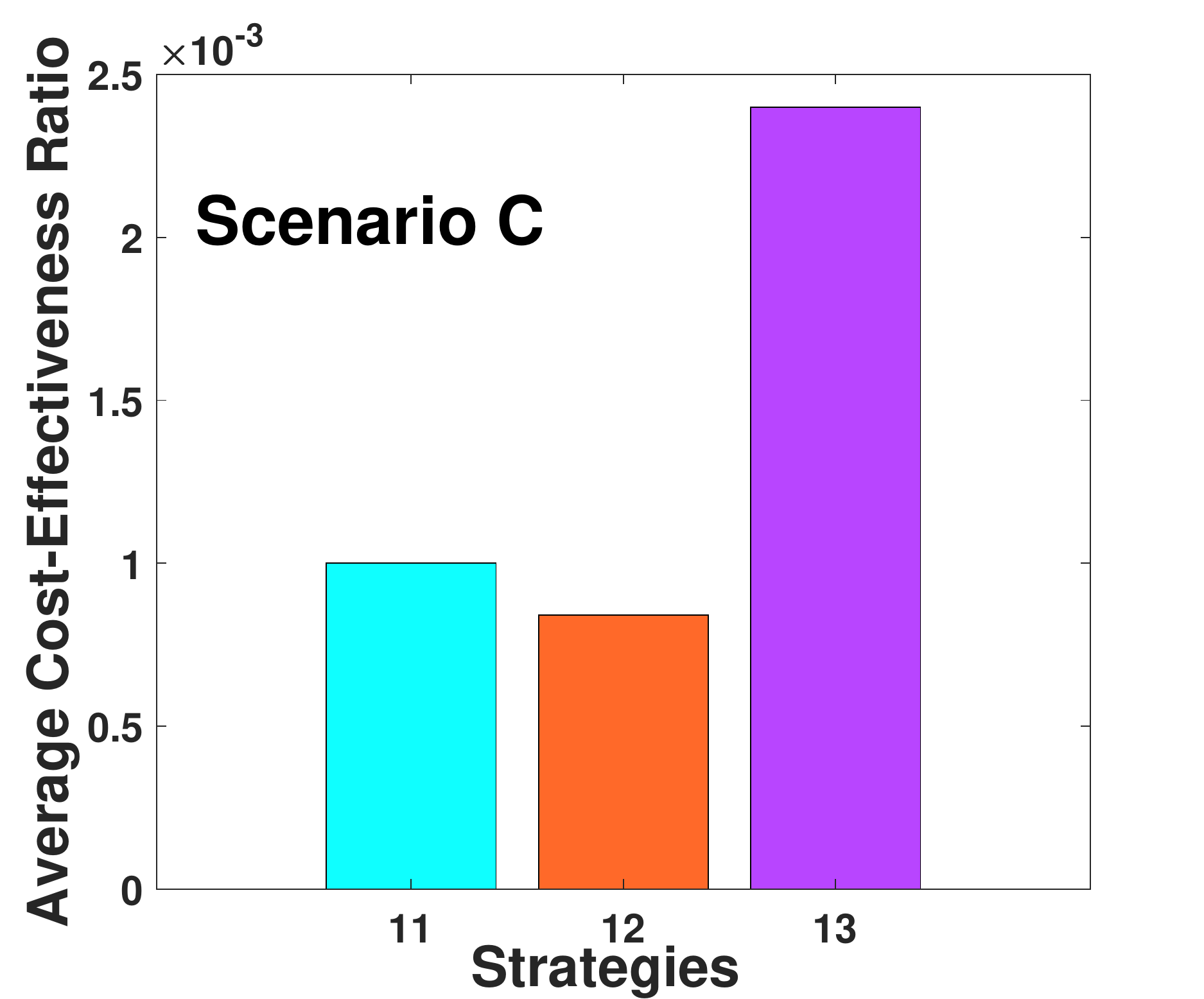}
		\caption{Average Cost-Effectiveness Ratio}\label{SC6}
	\end{subfigure}
	\caption{Implementation of quadruplet controls.}\label{3control}
\end{figure}
 In Figure \ref{SC3}, we noticed that the strategy with the highest number of virus removal from the environment is strategy 13 and 12, with strategy 11 having the most minimum virus removal effect. In Figure \ref{SC4}, the control profiles suggest that the optimal strategies for scenario C should be implemented on a control level of 0.75 for each control term in strategy 11, for 30 days and then reduced to 0.60 for the entire simulation period. For the control strategy 12 in Figure \ref{SC4}, we noticed that the control terms in strategy 12 should be kept at  0.7 for 18 days and then reduced to 0.5 for each of the control terms for the rest of the simulation time. The control profile for strategy 13 shows that each control term should be kept for 0.75 for 70 days and then reduced to 0.47 for the rest of the simulation period.  Figure \ref{SC5} shows the infection averted ratio (IAR) of the various control strategies. The IAR shows that strategy 11 is the most effective. Figure \ref{SC6} shows the average cost-effectiveness ratio (ACER), which indicates that strategy 12 is the most effective and cost-saving strategy in scenario C. The mathematical extraction of the infection averted ratio and the average cost-effectiveness ratio can be found in subsection \ref{costeffective}, where we validate the claim on Figure \ref{SC5} and \ref{SC6} respectively.

\subsubsection{Scenario D: use of quadruplet controls}
In Figure \ref{SD1}-\ref{SD4}, we carried out numerical simulations with the notion that an individual may apply all of the suggested controls simultaneously. We noticed in Figure \ref{SD1} that the number of exposed and asymptomatic individuals drastically reduces when the four controls are applied simultaneously. Figure \ref{SD2} shows that the disease in the symptomatic individuals can be eliminated within 21 days when one chooses to implement all the controls simultaneously.
\begin{figure}[h!]
	\centering
	\begin{subfigure}{.5\textwidth}
		\centering
		\includegraphics[width=1\linewidth, height=2in]{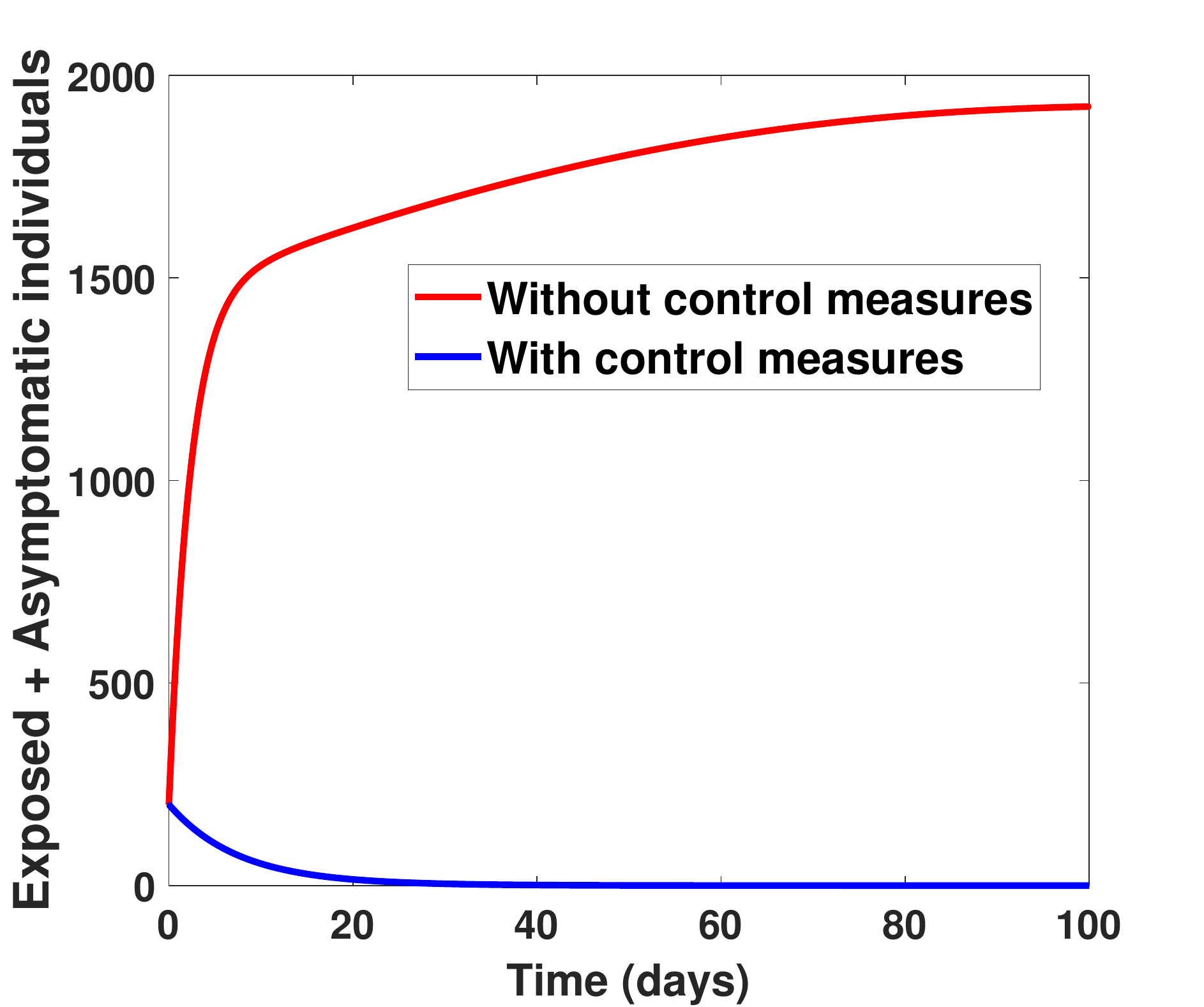}
		\caption{}\label{SD1}
	\end{subfigure}%
	\begin{subfigure}{.5\textwidth}
		\centering
		\includegraphics[width=1\linewidth, height=2in]{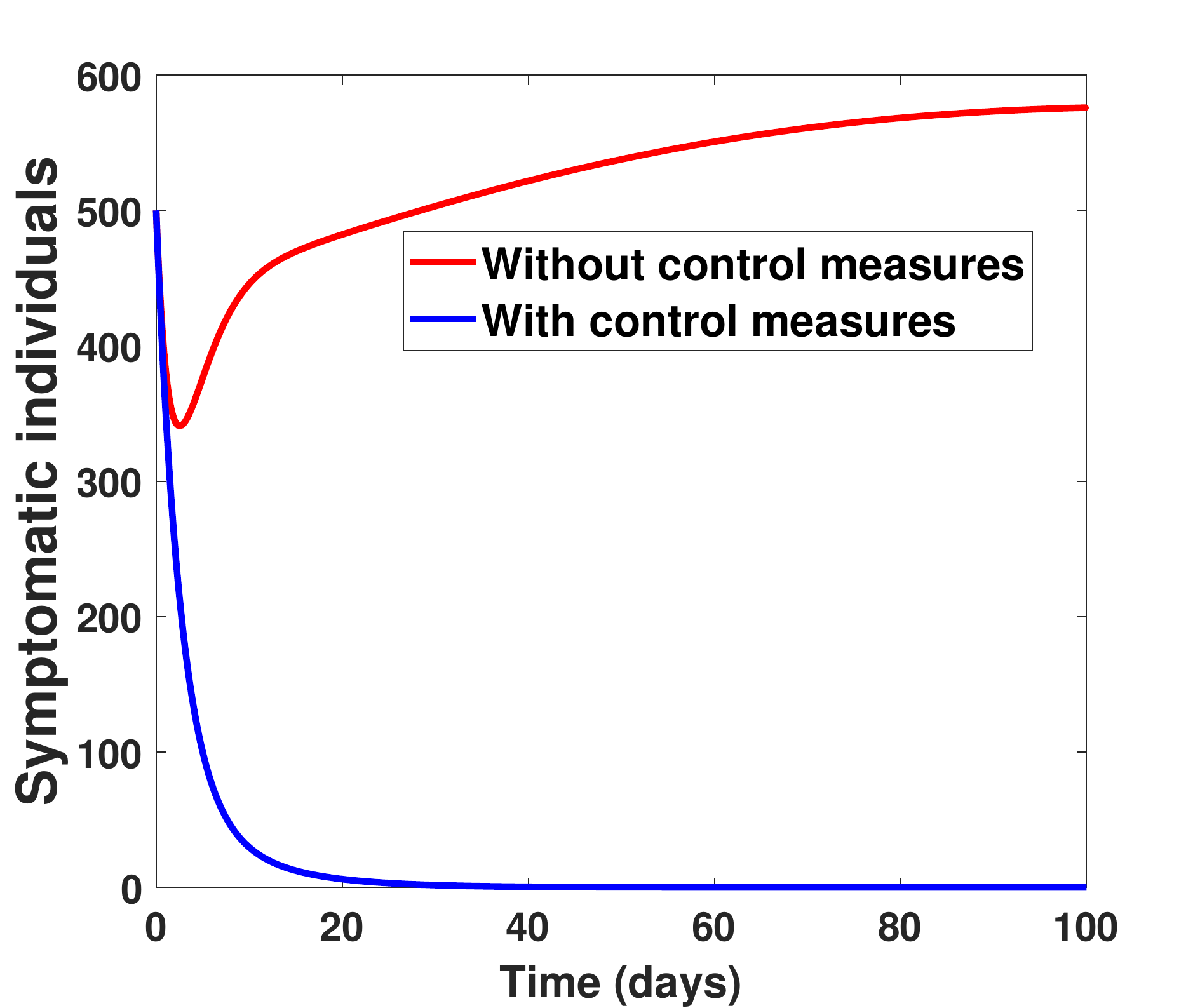}
		\caption{}\label{SD2}
	\end{subfigure}
	\begin{subfigure}{.5\textwidth}
		\centering
		\includegraphics[width=1\linewidth, height=2in]{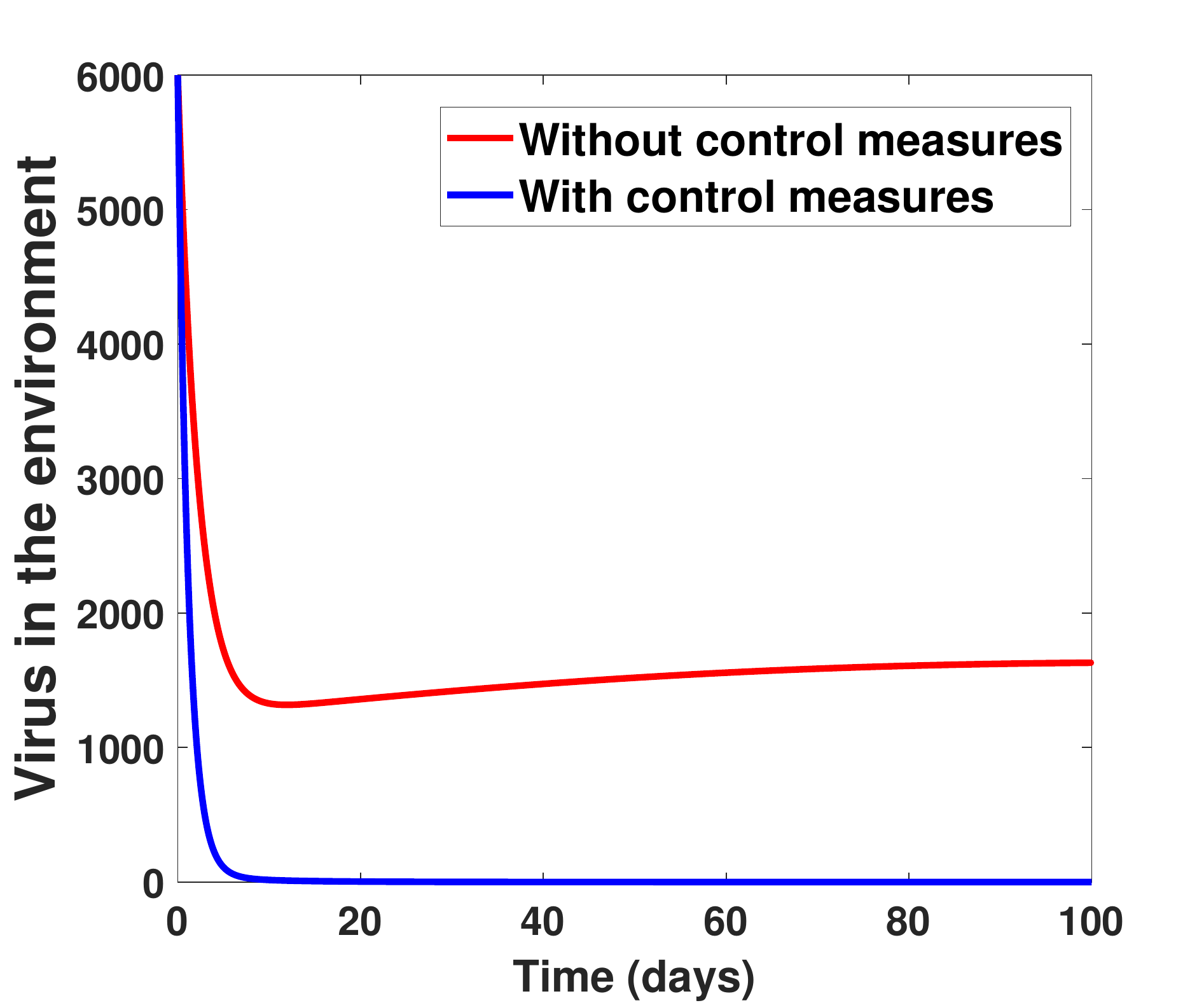}
		\caption{}\label{SD3}
	\end{subfigure}
\vspace{-0.4cm}
	\caption{With and without control strategies.}\label{Quadruplet}
\end{figure}
Figure \ref{SD3} shows that the virus in the environment can be eliminated within 10 days when one chooses to implement all the controls simultaneously. In Figure \ref{SD4} we showed the dynamical changes of each control considered in this work. We noticed that, in the pool of the four controls, control $u_1$ (practising physical or social distancing protocols) and control $u_2$ (practising personal hygiene by cleaning contaminated surfaces with alcohol-based detergents) should be applied at a constant level throughout, with much effort placed on control $u_1$ for 98 days. For the control $u_3$ in Figure \ref{SD4}, we noticed that the control should be kept at  0.75 for 25 days and then gradually reduced to 0.29 for the rest of the simulation time. The control profile for control $u_4$ shows that the control term should be kept for 0.75 for 18 days and then gradually reduced to 0.29 for the rest of the simulation period. Finally, Figure \ref{SD5} shows the efficacies plot for the number of exposed, asymptomatic, symptomatic individuals and the number of viruses removed from the environment, respectively, when one uses all the proposed control simultaneously. We noticed from the efficacies plot that the controls are more efficient on the number of viral removed from the environment, followed by the number of symptomatic individuals, asymptomatic individuals, and exposed individuals. The efficacy plots are obtained from using the following functions: $$E_{E} = \frac{E(0)-E^{*}(t)}{E(0)},  E_{I} = \frac{I(0)-I^{*}(t)}{I(0)}, E_{A} = \frac{A(0)-A^{*}(t)}{A(0)}, E_{B} = \frac{B(0)-B^{*}(t)}{B(0)}$$
where $E(0), I(0), A(0), B(0)$ are the initial data and $E^{*}(t),I^{*}(t), A^{*}(t),B^{*}(t)$ are the function relating to the ``optimal states associated" with the controls \cite{agusto2017optimal}. Figure \ref{SD5} shows that the controls attain $100\%$ efficacy on the disease induced compartment after 39 days.
\begin{figure}[h!]
	\centering
	\begin{subfigure}{.5\textwidth}
		\centering
		\includegraphics[width=1\linewidth,  height=2in]{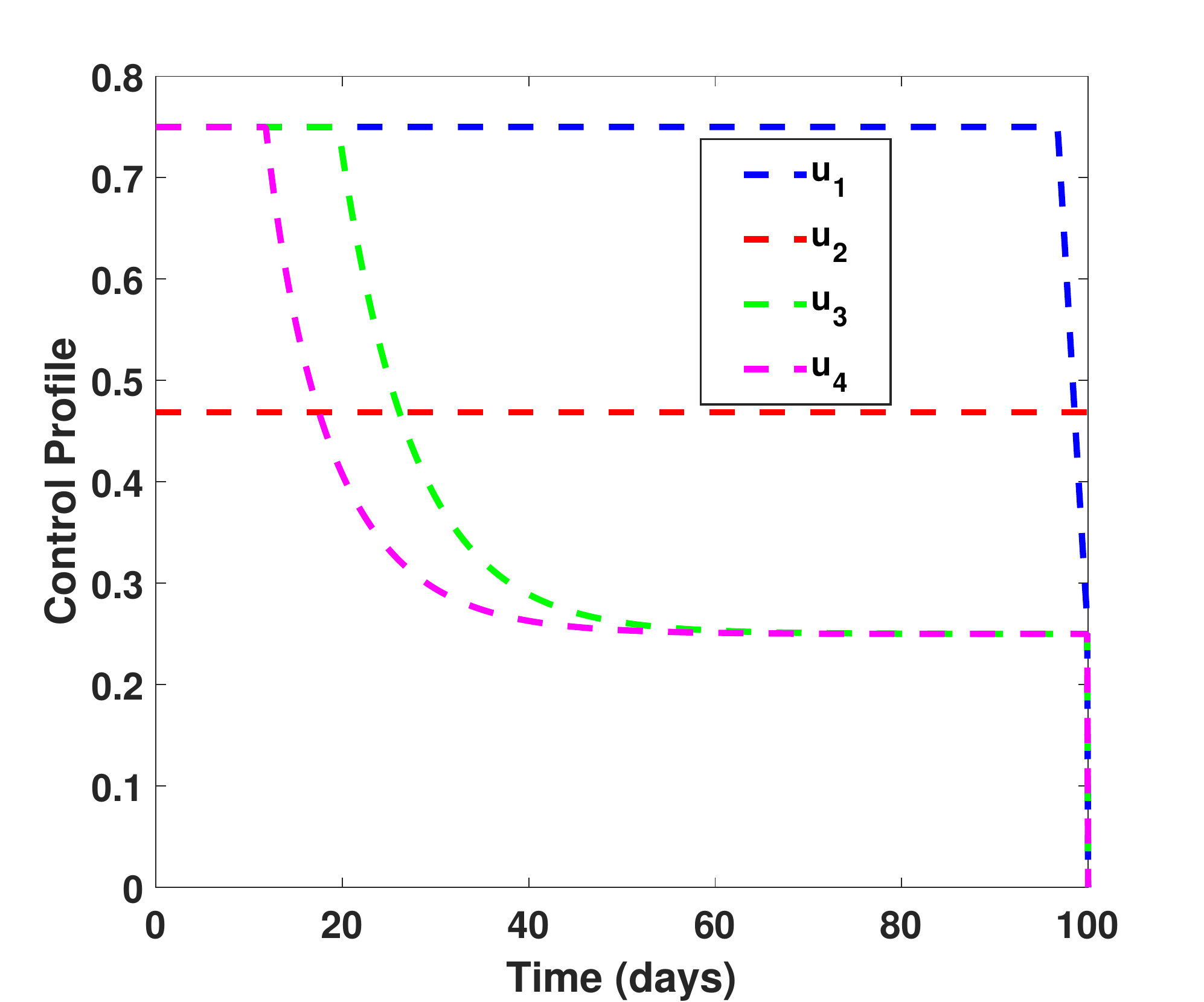}
		\caption{Control profile}\label{SD4}
	\end{subfigure}%
	\begin{subfigure}{.5\textwidth}
		\centering
		\includegraphics[width=1\linewidth, height=2in ]{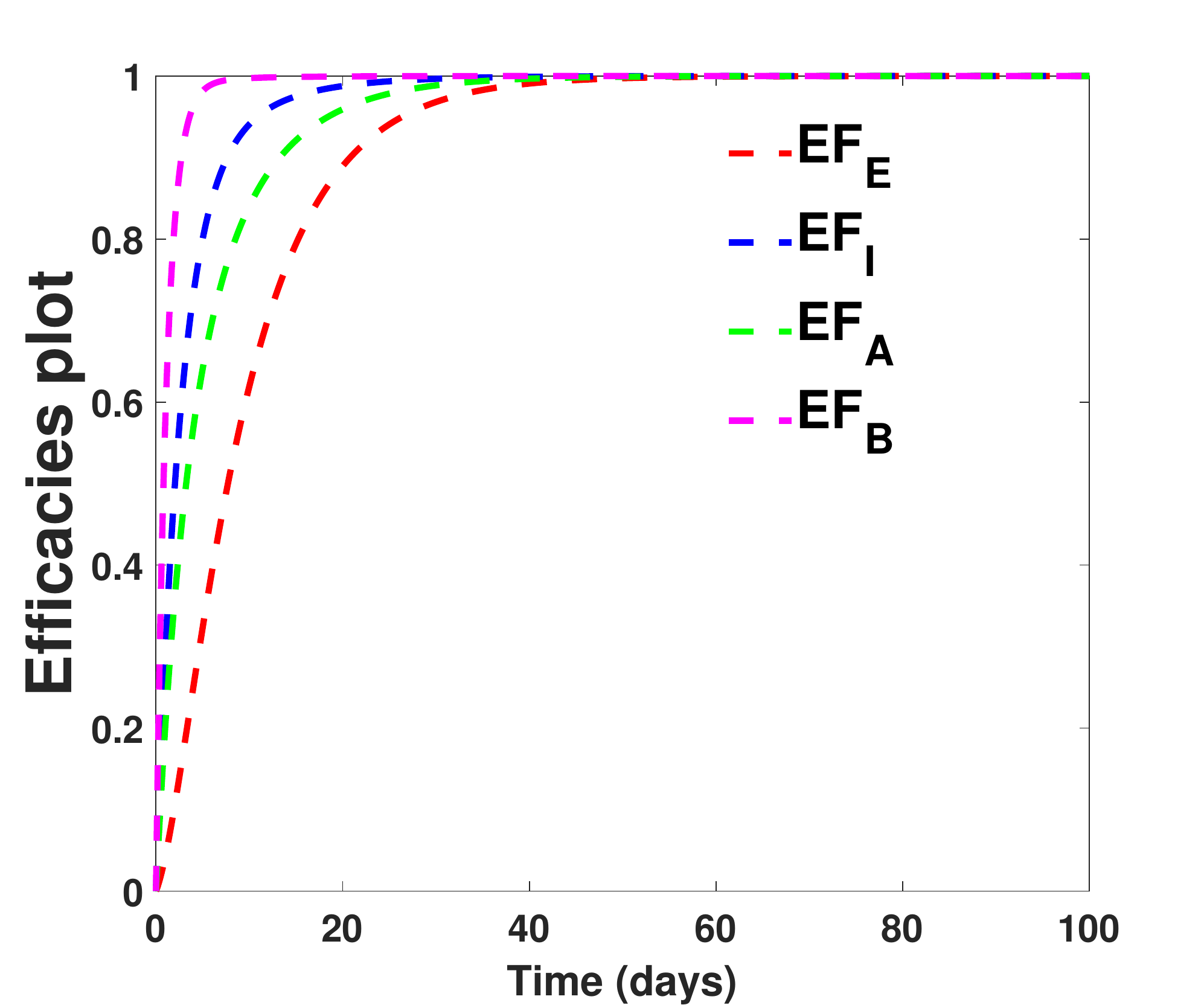}
		\caption{Efficacies  ratio}\label{SD5}
	\end{subfigure}
	\vspace{-0.4cm}
	\caption{With and without control strategies.}\label{different strategies}
\end{figure}
\subsection{Cost-effectiveness analysis}\label{costeffective}

Given the four different scenarios considered for the implementation of optimal control problem in section \ref{sec-numericsim}, cost-effectiveness analysis is employed to decide on the most cost-effective control intervention strategy from other strategies for each of scenarios A--D, under investigation. To implement the cost-effectiveness analysis, we use three approaches. These are: infection averted ratio (IAR) \cite{agusto2019optimal}, average cost-effectiveness ratio (ACER) and incremental cost-effectiveness ratio (ICER) \cite{agusto2019optimal,agusto2013optimal,agusto2017optimal}. Definitions of the three approaches are given as follows:

\subsubsection*{Infection averted ratio (IAR)}

Infection averted ratio (IAR) can be expressed as
\begin{equation*}
\mbox{IAR}=\frac{\mbox{Number of infections averted}}{\mbox{Number of individuals recovered from the infection}},
\end{equation*}
where the number of infections averted represents the difference between the total number of infected individuals without any control implementation and the total number of infected individuals with control throughout the simulation, a control strategy with the highest IAR value is considered as the most cost-effective \cite{agusto2019optimal,okyere2020analysis,agusto2013optimal}.

\subsubsection*{Average cost-effectiveness ratio (ACER)}

Average cost-effectiveness ratio (ACER) is stated as
\begin{equation*}
\mbox{ACER}=\frac{\mbox{Total cost incurred on the implementation of a particular intervention strategy}}{\mbox{Total number of infections averted by the intervention strategy}}.
\end{equation*}
The total cost incurred on implementing a particular intervention strategy is estimated from
\begin{equation}
\label{tcost}
\mathcal C(u)=\frac{1}{2}\int_0^T\sum_{i=1}^4D_iu_i^2dt.
\end{equation}

\subsubsection*{Incremental cost-effectiveness ratio (ICER)}

Usually, the incremental cost-effectiveness ratio (ICER) measures the changes between the costs and health benefits of any two different intervention strategies competing for the same limited resources. Considering strategies p and q as two competing control intervention strategies, then ICER is stated as
\begin{equation*}
\mbox{ICER}=\frac{\mbox{Change in total costs in strategies p and q}}{\mbox{Change in control benefits in strategies p and q}}.
\end{equation*}
ICER numerator includes the differences in disease averted costs, costs of prevented cases, intervention costs, among others. While the denominator of ICER accounts for the differences in health outcome, including the total number of infections averted or the total number of susceptibility cases prevented.

\subsubsection{Scenario A: use of single control}

Owing to the simulated results of the optimality system under scenario A (when only one control is implemented with considerations of strategies 1--4) as shown in Figure \ref{Singlecontrol}, we calculate IAR, ACER and ICER for each of the four control strategies.

For IAR, the fourth column of Table \ref{ICERT1} summarizes the calculated values for the implemented strategies. Accordingly, strategy 2 (practising personal hygiene by cleaning contaminated surfaces with alcohol-based detergents only) has the highest IAR value, followed by strategy 1 (practising physical or social distancing protocol only), strategy 4 (fumigating schools in all levels of education, sports facilities and commercial areas such as markets and public toilet facilities only), and lastly strategy 3 (practising proper and safety measures by the exposed, asymptomatic infected and symptomatic infected individuals only). Consequently, the most cost-effective strategy according to this cost-effectiveness analysis approach is strategy 2. The next most cost-effective strategy is strategy 1, followed by strategy 4, then strategy 3.

According to the ACER cost-effectiveness analysis method, strategy 4 has the highest ACER value, followed by strategy 3, strategy 2 and strategy 1 as shown in the fifth column of Table \ref{ICERT1}. Therefore, the cost-effectiveness of the four strategies implemented, ranging from the most cost-effective to the least cost-effective strategy, is given as strategy 1, strategy 2, strategy 3, and strategy 4.

Next, ICER values are computed for the four control intervention strategies under scenario A to further affirm the most economical strategy among them. Based on the results obtained for the numerical simulations of optimal control problem in scenario A (see Figure \ref{Singlecontrol}), strategies 1--4 are ranked according to their increasing order in respect of the total number of COVID-19 infections averted in the community. We have that Strategy 3 averts the least number of the disease infections, followed by Strategy 2, Strategy 4 and Strategy 1 as shown in Table \ref{ICERT1}.

\begin{table}[h!]
	\centering
	\caption{Incremental cost-effectiveness ratio for scenario A \label{ICERT1}}
	\begin{tabular}{lccccc}
		\toprule
		{\bf Strategy} & {\bf Infection averted} & {\bf Cost} & {\bf IAR} & {\bf ACER}& {\bf ICER}\\ \toprule
		Strategy 3: $u_3(t)$ & $1.4423\times 10^{6}$ & $1.4063\times 10^{3}$ & $1.2325$ & $9.7498\times 10^{-4}$&$9.7498\times 10^{-4}$\\
		Strategy 2: $u_2(t)$ & $1.6603\times 10^{6}$ & $1.4077\times 10^{3}$ & $1.5835$ & $8.4784\times 10^{-4}$&$6.4220\times 10^{-6}$\\
		Strategy 4: $ u_4(t)$ & $1.8000\times 10^{6}$ & $2.8098\times 10^{3}$ & $1.2914$ & $0.0016$&0.0100\\
		Strategy  1: $u_1(t),$ & $2.0679\times 10^{6}$ & $281.1135$ & $1.5793$ & $1.3594\times 10^{-4}$&$-0.0004$\\
		\bottomrule
	\end{tabular}
\end{table}

Thus, ICER is computed for the competing control Strategy 1, Strategy 2, Strategy 3 and Strategy 4 as follows:
\begin{align*}
	\mbox{ICER(3)}&=\frac{1.4063\times 10^3-0}{1.4423\times10^6-0}= 9.7504\times 10^{-4},\\
	\mbox{ICER(2)}&=\frac{1.4077\times 10^3 - 1.4063\times 10^3}{1.6603\times10^6 -1.4423\times10^6}=6.4220\times 10^{-6},\\
	\mbox{ICER(4)}&=\frac{2.8098\times 10^3 - 1.4077\times 10^3 }{1.8000\times 10^6-1.6603\times10^6  }=0.0100,\\
	\mbox{ICER(1)}&=\frac{281.1135-2.8098\times 10^3}{2.0679\times 10^6-1.8000\times 10^6 }=-0.0004.
\end{align*}

The computed results (as presented in Table \ref{ICERT1}) indicate that the ICER value of Strategy 4, ICER(4), is higher than that of Strategy 3. This means that the singular application of control $u_4$ (fumigating schools in all levels of education, sports facilities and commercial areas such as markets and public toilet facilities) is more costly and less effective than when only control $u_3$ (practising proper and safety measures by the exposed, asymptomatic infected and symptomatic infected individuals) is applied. Thus, Strategy 4 is eliminated from the list of alternative control strategies.

Then, ICER is further calculated for the competing Strategy 3 with Strategies 1 and 2. The computation is as follows:
\begin{align*}
	\mbox{ICER(3)}&=\frac{1.4063\times 10^3-0}{1.4423\times10^6-0}= 9.7504\times 10^{-4},\\
	\mbox{ICER(2)}&=\frac{1.4077\times 10^3 - 1.4063\times 10^3}{1.6603\times10^6 -1.4423\times10^6}=6.4220\times 10^{-6},\\
	\mbox{ICER(1)}&=\frac{281.1135-1.4077\times 10^3}{2.0679\times 10^6-1.6603\times10^6  }=-0.0028.
\end{align*}
The summary of ICER calculations is summarized in Table \ref{ICERT2}. Looking at Table \ref{ICERT2}, it is seen that there is a cost-saving of $6.4220\times10^{-6}$ for Strategy 2 over Strategy 3 following the comparison of ICER(2) and ICER(3). The obtained lower ICER for Strategy 2 indicates that Strategy 3 strongly dominates Strategy 2, implying that Strategy 3 is more costly and less effective to implement than Strategy 2. Thus, it is better to eliminate Strategy 3 from the control intervention strategies and focus on the alternative control interventions to implement for limited resources preservation. Consequently, Strategy 3 is excluded, and Strategy 2 is further compared with Strategy 1.

\begin{table}[h!]
	\centering
	\caption{Incremental cost-effectiveness ratio for scenario A \label{ICERT2}}
	\begin{tabular}{lccccc}
		\toprule
		{\bf Strategy} & {\bf Infection averted} & {\bf Cost} & {\bf IAR} & {\bf ACER}& {\bf ICER}\\ \toprule
		Strategy 3: $u_3(t)$ & $1.4423\times 10^{6}$ & $1.4063\times 10^{3}$ & $1.2325$ & $9.7498\times 10^{-4}$&$9.7504\times 10^{-4}$\\
		Strategy 2: $u_2(t)$ & $1.6603\times 10^{6}$ & $1.4077\times 10^{3}$ & $1.5835$ & $8.4784\times 10^{-4}$&$6.4220\times 10^{-6}$\\
		Strategy  1: $u_1(t),$ & $2.0679\times 10^{6}$ & $281.1135$ & $1.5793$ & $1.3594\times 10^{-4}$&$-0.0028$\\
		\bottomrule
	\end{tabular}
\end{table}

We now face the re-calculation of the ICER for Strategies 1 and 2. The calculations are made as follows:
\begin{align*}
	\mbox{ICER(2)}&=\frac{1.4077\times 10^3 -0}{1.6603\times10^6 -0}=8.4784\times 10^{-4},\\
	\mbox{ICER(1)}&=\frac{281.1135-1.4077\times 10^3}{2.0679\times 10^6-1.6603\times10^6  }=-0.0028.
\end{align*}
The results obtained from ICER computations are presented in Table \ref{ICERT3}. From Table \ref{ICERT3}, it is shown that ICER(2) is greater than ICER(1). The implication of the lower ICER value obtained for Strategy 1 is that Strategy 2 strongly dominates, implying that Strategy 2 is more costly and less effective to implement than Strategy 1. Therefore, Strategy 1 (practising physical or social distancing protocol only) is considered the most cost-effective among the four strategies in Scenario A analysed in this work, which confirms the results in Figure \ref{SA6}.

\begin{table}[h!]
	\centering
	\caption{Incremental cost-effectiveness ratio for scenario A \label{ICERT3}}
	\begin{tabular}{lccccc}
		\toprule
		{\bf Strategy} & {\bf Infection averted} & {\bf Cost} & {\bf IAR} & {\bf ACER}& {\bf ICER}\\ \toprule
		Strategy 2: $u_2(t)$ & $1.6603\times 10^{6}$ & $1.4077\times 10^{3}$ & $1.5835$ & $8.4784\times 10^{-4}$&$8.4784\times 10^{-4}$\\
		Strategy  1: $u_1(t),$ & $2.0679\times 10^{6}$ & $281.1135$ & $1.5793$ & $1.3594\times 10^{-4}$&$-0.0028$\\
		\bottomrule
	\end{tabular}
\end{table}

\subsubsection{Scenario B: use of double controls}

According to the results obtained from the numerical implementation of the optimality system under Scenario B (when only two different controls are implemented with considerations of Strategies 5--10) as illustrated in Figure \ref{Doublecontrol}, we discuss the IAR, ACER and ICER cost analysis techniques for Strategies 5--10 here.

To compare Strategies 5--10 using the IAR cost analysis approach, the computed values for the six control strategies are as presented in the fourth column of Table \ref{ICERB1}. A look at Table \ref{ICERB1} shows that Strategy 5 has the highest IAR. This is followed by Strategy 6, then Strategies 8, 7, 9 and 10. Therefore, it follows that Strategy 5 (which combines practising physical or social distancing protocols with practising personal hygiene by cleaning contaminated surfaces with alcohol-based detergents) is considered most cost-effective among the six strategies in Scenario B as analysed according to the IAR cost analysis technique.

Also, we use the ACER technique to determine the most cost-effective strategy among the various intervention strategies considered in Scenario B. From the results obtained (as shown contained in the fifth column of Table \ref{ICERB1}), it is clear that Strategy 6 has the least ACER value, followed by Strategies 5, 7, 8, 9 and 10. Hence, Strategy 6 (which combines practising physical or social distancing protocols with practising proper and safety measures by exposed, asymptomatic infected and asymptomatic infected individuals) is the most cost-effective among the set of control strategies considered in Scenario B based on the ACER cost-effective analysis method.

To further affirm the most cost-effective strategy among Strategies 5--10, we implement ICER cost analysis approach on the six intervention strategies. Using the simulated results (as demonstrated in Figure \ref{Doublecontrol}), the six control strategies are ranked from least to most effective according to the number of COVID-19 infections averted as shown in Table \ref{ICERB1}. So, Strategy 8 averts the least number of infections, followed by Strategy 10, Strategy 9, Strategy 6, Strategy 7 and Strategy 5, averting the most number of infections in the population.

\begin{table}[h!]
	\centering
	\caption{Incremental cost-effectiveness ratio for scenario B\label{ICERB1}}
	\begin{tabular}{lccccc}
		\toprule
		{\bf Strategy} & {\bf Infection averted} & {\bf Cost} & {\bf IAR} & {\bf ACER}& {\bf ICER} \\ \toprule
		Strategy 8: $u_2(t), u_3(t)$ & $1.9253\times 10^{6}$ & $2.8104\times 10^{3}$ & $1.4524$ & $0.0015$&$0.0015$\\
		Strategy 10: $u_3(t), u_4(t)$ & $1.9809\times 10^{6}$ & $4.1019\times 10^{3}$ & $ 1.3350$ & $0.0021$&$0.0232$\\
		Strategy 9: $ u_2(t), u_4(t)$ & $2.0684\times 10^{6}$ & $4.1495\times 10^{3}$ & $1.4077$ & $0.0020$&$5.4400\times10^{-4}$\\
		Strategy 6: $u_1(t), u_3(t)$ & $2.1128\times 10^{6}$ & $1.3487\times 10^{3}$ & $1.5239$ & $6.3834\times 10^{-4}$&$-0.0631$\\
		Strategy 7: $u_1(t), u_4(t)$ & $2.1751\times 10^{6}$ & $1.7708\times 10^{3}$ & $1.4506$ & $8.1410\times 10^{-4}$&$0.0068$\\
		Strategy 5: $u_1(t), u_2(t)$ & $2.2265\times 10^{6}$ & $1.6871\times 10^{3}$ & $1.5759$ & $7.5775\times 10^{-4}$&$-0.0016$\\
		\bottomrule
	\end{tabular}
\end{table}
The ICER value for each strategy is computed as follows:
\begin{align*}
	\mbox{ICER(8)}&=\frac{2.8104\times 10^3-0}{1.9253\times10^6-0}= 0.0015,\\
	\mbox{ICER(10)}&=\frac{4.1019\times 10^3 - 2.8104\times 10^3}{1.9809\times10^6 -1.9253\times10^6}=0.0232,\\
	\mbox{ICER(9)}&=\frac{4.1495\times 10^3 - 4.1019\times 10^3 }{2.0684\times 10^6-1.9809\times10^6  }=5.4400\times10^{-4},\\
	\mbox{ICER(6)}&=\frac{1.3487\times10^3-4.1495\times 10^3}{2.1128\times 10^6-2.0684\times 10^6 }=-00631,\\
	\mbox{ICER(7)}&=\frac{1.7708\times 10^3 - 1.3487\times 10^3 }{2.1751\times 10^6-2.1128\times10^6  }=0.0068,\\
	\mbox{ICER(5)}&=\frac{1.6871\times10^{3}-1.7708\times 10^3}{2.2265\times 10^{6}-2.1751\times 10^6 }=-0.0016.
\end{align*}
From Table \ref{ICERB1}, it is observed that there is a cost-saving of \$0.0068 for Strategy 7 over Strategy 10. This follows the comparison of ICER(7) and ICER(10). The indication of the lower ICER value obtained for Strategy 7 is that Strategy 10 strongly dominates Strategy 7. By implication, Strategy 10 is more costly and less effective to implement when compared with Strategy 7. Therefore, it is better to exclude Strategy 10 from the set of alternative intervention strategies. At this point, Strategy 7 is compared with Strategies 5, 6, 8 and 9.

The ICER is computed as
\begin{align*}
	\mbox{ICER(8)}&=\frac{2.8104\times 10^3-0}{1.9253\times10^6-0}= 0.0015,\\
	\mbox{ICER(9)}&=\frac{4.1495\times 10^3 - 2.8104\times 10^3 }{2.0684\times 10^6-1.9253\times10^6  }=0.0094,\\
	\mbox{ICER(6)}&=\frac{1.3487\times10^3-4.1495\times 10^3}{2.1128\times 10^6-2.0684\times 10^6 }=-00631,\\
\mbox{ICER(7)}&=\frac{1.7708\times 10^3 - 1.3487\times 10^3 }{2.1751\times 10^6-2.1128\times10^6  }=0.0068,\\
	\mbox{ICER(5)}&=\frac{1.6871\times10^{3}-1.7708\times 10^3}{2.2265\times 10^{6}-2.1751\times 10^6 }=-0.0016.
\end{align*}
The results obtained are summarized in Table \ref{ICERB2}.
\begin{table}[h!]
	\centering
	\caption{Incremental cost-effectiveness ratio for scenario B\label{ICERB2}}
	\begin{tabular}{lccccc}
		\toprule
		{\bf Strategy} & {\bf Infection averted} & {\bf Cost} & {\bf IAR} & {\bf ACER}& {\bf ICER} \\ \toprule
		Strategy 8: $u_2(t), u_3(t)$ & $1.9253\times 10^{6}$ & $2.8104\times 10^{3}$ & $1.4524$ & $0.0015$&$0.0015$\\
			Strategy 9: $ u_2(t), u_4(t)$ & $2.0684\times 10^{6}$ & $4.1495\times 10^{3}$ & $1.4077$ & $0.0020$&$0.0094$\\
		Strategy 6: $u_1(t), u_3(t)$ & $2.1128\times 10^{6}$ & $1.3487\times 10^{3}$ & $1.5239$ & $6.3834\times 10^{-4}$&$-0.0631$\\
		Strategy 7: $u_1(t), u_4(t)$ & $2.1751\times 10^{6}$ & $1.7708\times 10^{3}$ & $1.4506$ & $8.1410\times 10^{-4}$&$0.0068$\\
		Strategy 5: $u_1(t), u_2(t)$ & $2.2265\times 10^{6}$ & $1.6871\times 10^{3}$ & $1.5759$ & $7.5775\times 10^{-4}$&$-0.0016$\\
		\bottomrule
	\end{tabular}
\end{table}
Table \ref{ICERB2} shows a cost-saving of \$0.0068 for Strategy 7 over Strategy 9 by comparing ICER(7) and ICER(9). The higher ICER value obtained for Strategy 9 implies that Strategy 9 strongly dominated, more costly and less effective to implement when compared with Strategy 7. Therefore, Strategy 9 is left out of the list of alternative control interventions to implement for the purpose of preserving the limited resources. We further compare Strategy 7 with Strategies 5, 6 and 8.

The computation of ICER for Strategies 5, 6, 7 and 8 is as follows:
\begin{align*}
	\mbox{ICER(8)}&=\frac{2.8104\times 10^3-0}{1.9253\times10^6-0}= 0.0015,\\
	\mbox{ICER(6)}&=\frac{1.3487\times10^3-2.8104\times 10^3}{2.1128\times 10^6-1.9253\times 10^6 }= -0.0078,\\
\mbox{ICER(7)}&=\frac{1.7708\times 10^3 - 1.3487\times 10^3 }{2.1751\times 10^6-2.1128\times10^6  }=0.0068,\\
	\mbox{ICER(5)}&=\frac{1.6871\times10^{3}-1.3487\times 10^3}{2.2265\times 10^{6}-2.1128\times 10^6 }=0.0030.
\end{align*}
The summary of the results obtained is presented in Table \ref{ICERB3}.
\begin{table}[h!]
	\centering
	\caption{Incremental cost-effectiveness ratio for scenario B\label{ICERB3}}
	\begin{tabular}{lccccc}
		\toprule
		{\bf Strategy} & {\bf Infection averted} & {\bf Cost} & {\bf IAR} & {\bf ACER}& {\bf ICER} \\ \toprule
		Strategy 8: $u_2(t), u_3(t)$ & $1.9253\times 10^{6}$ & $2.8104\times 10^{3}$ & $1.4524$ & $0.0015$&$0.0015$\\
			Strategy 6: $u_1(t), u_3(t)$ & $2.1128\times 10^{6}$ & $1.3487\times 10^{3}$ & $1.5239$ & $6.3834\times 10^{-4}$&$ -0.0078$\\
		Strategy 7: $u_1(t), u_4(t)$ & $2.1751\times 10^{6}$ & $1.7708\times 10^{3}$ & $1.4506$ & $8.1410\times 10^{-4}$&$0.0068$\\
		Strategy 5: $u_1(t), u_2(t)$ & $2.2265\times 10^{6}$ & $1.6871\times 10^{3}$ & $1.5759$ & $7.5775\times 10^{-4}$&$-0.0016$\\
	\bottomrule
	\end{tabular}
\end{table}

Looking at Table \ref{ICERB3}, a comparison of ICER(7) and ICER(8) shows a cost-saving of \$0.0015 for Strategy 8 over Strategy 7. The lower ICER obtained for Strategy 8 is that Strategy 7 strongly dominated, more costly and less effective to implement than Strategy 8. Thus, it is better to discard Strategy 7 from the list of alternative intervention strategies. At this juncture, Strategy 8 is further compared with Strategies 5 and 6.

The calculation of ICER is given as
\begin{align*}
	\mbox{ICER(8)}&=\frac{2.8104\times 10^3-0}{1.9253\times10^6-0}= 0.0015,\\
	\mbox{ICER(6)}&=\frac{1.3487\times10^3-2.8104\times 10^3}{2.1128\times 10^6-1.9253\times 10^6 }=-0.0078,\\
	\mbox{ICER(5)}&=\frac{1.6871\times10^{3}-1.3487\times 10^3}{2.2265\times 10^{6}-2.1128\times 10^6 }=0.0030.
\end{align*}
Table \ref{ICERB4} summarizes the results obtained from the ICER computations.
\begin{table}[h!]
	\centering
	\caption{Incremental cost-effectiveness ratio for scenario B\label{ICERB4}}
	\begin{tabular}{lccccc}
		\toprule
		{\bf Strategy} & {\bf Infection averted} & {\bf Cost} & {\bf IAR} & {\bf ACER}& {\bf ICER} \\ \toprule
		Strategy 8: $u_2(t), u_3(t)$ & $1.9253\times 10^{6}$ & $2.8104\times 10^{3}$ & $1.4524$ & $0.0015$&$0.0015$\\
		Strategy 6: $u_1(t), u_3(t)$ & $2.1128\times 10^{6}$ & $1.3487\times 10^{3}$ & $1.5239$ & $6.3834\times 10^{-4}$&$-0.0078$\\
		Strategy 5: $u_1(t), u_2(t)$ & $2.2265\times 10^{6}$ & $1.6871\times 10^{3}$ & $1.5759$ & $7.5775\times 10^{-4}$&$0.0030$\\
		\bottomrule
	\end{tabular}
\end{table}
In Table \ref{ICERB4}, it is shown that there is a cost-saving of \$0.0015 for Strategy 8 over Strategy 5 following the comparison of ICER(5) with ICER(8). The higher ICER value for Strategy 5 suggests that Strategy 5 is dominated, more costly and less effective to implement than Strategy 8. Hence, Strategy 5 is discarded from the set of alternative intervention strategies. Finally, Strategy 8 is compared with Strategy 6.
The ICER is computed as follows:
\begin{align*}
	\mbox{ICER(8)}&=\frac{2.8104\times 10^3-0}{1.9253\times10^6-0}= 0.0015,\\
	\mbox{ICER(6)}&=\frac{1.3487\times10^3-2.8104\times 10^3}{2.1128\times 10^6-1.9253\times 10^6 }=-0.0078.
\end{align*}
We give the summary of the results in Table \ref{ICERB5}.
\begin{table}[h!]
	\centering
	\caption{Incremental cost-effectiveness ratio for scenario B\label{ICERB5}}
	\begin{tabular}{lccccc}
		\toprule
		{\bf Strategy} & {\bf Infection averted} & {\bf Cost} & {\bf IAR} & {\bf ACER}& {\bf ICER} \\ \toprule
		Strategy 8: $u_2(t), u_3(t)$ & $1.9253\times 10^{6}$ & $2.8104\times 10^{3}$ & $1.4524$ & $0.0015$&$0.0015$\\
		Strategy 6: $u_1(t), u_3(t)$ & $2.1128\times 10^{6}$ & $1.3487\times 10^{3}$ & $1.5239$ & $6.3834\times 10^{-4}$&$-0.0078$\\
		\bottomrule
	\end{tabular}
\end{table}

Table \ref{ICERB5} reveals that ICER(8) is greater than ICER(6), implying that Strategy 6 is dominated by Strategy 8. This indicates that Strategy 8 is more costly and less effective to implement when compared with Strategy 6. Therefore, Strategy 8 is excluded from the list of alternative intervention strategies. Consequently, Strategy 6 (which combines practising physical or social distancing protocols with practising proper and safety measures by exposed, asymptomatic infected and asymptomatic infected individuals) is considered most cost-effective among the six different control strategies in Scenario B under investigation in this study, which confirms the results in Figure \ref{SB6}.

\subsubsection{Scenario C: use of triple controls}

This part explores the implementation of IAR, ACER and ICER cost analysis techniques on Strategies 11, 12 and 13 using the results obtained from the numerical simulations of the optimality system under Scenario C as presented in Figure \ref{3control}.

To determine the most cost-effective strategy among strategies 11, 12 and 13 using the IAR method, the obtained IAR values for the three strategies are given in the fourth column of Table \ref{ICERC1}. It is shown that Strategy 11 has the highest IAR value, followed by Strategy 12, then Strategy 13, which has the lowest IAR value. Therefore, based on this cost analysis approach, Strategy 11 (which combines practising physical or social distancing protocols with the efforts of practising personal hygiene by cleaning contaminated surfaces with alcohol-based detergents and practising proper and safety measures by exposed, asymptomatic infected and asymptomatic infected individual) is the most cost-effective control strategy to implement in Scenario C.

Also, the ACER cost analysis approach is employed to determine the most cost-effective strategy among Strategies 11, 12 and 13. To do this, the ACER values obtained for these strategies are as given in the fifth column of Table \ref{ICERC1}. It is observed that Strategy 12 has the lowest ACER value. The successive strategy with the lowest ACER value is Strategy 11, followed by Strategy 13, which has the highest ACER value. Therefore, according to ACER cost analysis, Strategy 12 is the most cost-effective Strategy to implement in Scenario C.

\begin{table}[h!]
	\centering
	\caption{Incremental cost-effectiveness ratio for scenario C\label{ICERC1}}
	\begin{tabular}{lccccc}
		\toprule
		{\bf Strategy} & {\bf Infection averted} & {\bf Cost} & {\bf IAR} & {\bf ACER}& {\bf ICER} \\ \toprule
		Strategy 13: $ u_2(t),u_3(t), u_4(t)$ & $2.1053\times 10^{6}$ & $5.0186\times 10^{3}$ & $1.4022$ & $0.0024$& $0.0024$\\
		Strategy 11: $u_1(t), u_2(t),u_3(t)$ & $2.2265\times 10^{6}$ & $2.2464\times 10^{3}$ & $1.5641$ & $0.0010$&$-0.0229$\\
		Strategy 12: $u_1(t), u_2(t), u_4(t)$ & $2.2265\times 10^{6}$ & $1.8726\times 10^{3}$ & $1.4706$ & $8.4107\times 10^{-4}$&$-$\\
		\bottomrule
	\end{tabular}
\end{table}

The cost-effective strategy among Strategies 11, 12 and 13 is considered in Scenario C using ICER and cost-minimizing analysis technique due to the equal number of infection averted by Strategies 11, 12. To implement this technique, the three intervention strategies are ranked in increasing order based on the total number of COVID-19 infections averted.

The calculation of ICER in Table \ref{ICERC1} is demonstrated as follows:
\begin{align*}
	\mbox{ICER(13)}&=\frac{5.0186\times 10^3}{2.1053\times10^6}= 0.0024,\\
	\mbox{ICER(11)}&=\frac{2.2464\times10^3-5.0186\times 10^3}{2.2265\times 10^6-2.1053\times 10^6 }=-0.0229.
\end{align*}
Note that, due to the equal number of infection averted by Strategies 11, 12, the ICER is not compared between these strategies. It is shown in Table \ref{ICERC1} that ICER(13) is greater than ICER(11). Thus, Strategy 13 strongly dominates Strategy 11, implying that Strategy 13 is more costly and less effective to implement in comparison with Strategy 11. Therefore, Strategy 13 is eliminated from the list of alternative control strategies. At this point, there is no need to re-compute ICER further for the competing Strategies 11 and 12 because the two strategies avert the same total number of infections. However, the minimization cost technique is used to decide which of the strategies is more cost-effective. It is seen that Strategy 12 requires a lower cost to be implemented compared to Strategy 11. Therefore, Strategy 12 (which combines practising physical or social distancing protocols with the efforts of practising personal hygiene by cleaning contaminated surfaces with alcohol-based detergents and Fumigating schools in all levels of education, sports facilities and commercial areas such as markets and public toilet facilities) is considered the most cost-effective strategy in Scenario C.

\subsubsection{Scenario D: implementation of quadruplet}
Using the simulated results for the optimality system when Strategy 14 in Scenario D is implemented (see Figure \ref{Quadruplet}), the cost-effective analysis of this Strategy based on IAR, ACER shown.

Table \ref{ICERD} gives the summary of the results obtained from implementing the IAR and ACER cost analysis techniques.
\begin{table}[h!]
	\centering
	\caption{Application of optimal controls: scenario D\label{ICERD}}
	\begin{tabular}{lcccc}
		\toprule
		{\bf Strategy} & {\bf Infection averted} & {\bf Cost} & {\bf IAR} & {\bf ACER} \\ \toprule
		Strategy 14: $u_1(t), u_2(t),u_3(t), u_4(t)$ & $2.2265\times 10^{6}$ & $2.0437\times 10^{3}$ & $1.4662$ & $9.1789\times 10^{-4}$\\
		\bottomrule
	\end{tabular}
\end{table}

\subsubsection{Determination of the overall most cost-effective strategy}

So far, we have been able to obtain the most cost-effective strategy corresponding to each of the four scenarios considered in this study and also noticed from Figure \ref{Quadruplet} that using all the controls reduces the disease faster. Hence, it is also essential to determine the most cost-effective strategy from the four most cost-effective strategy corresponding to a particular scenario. Thus, IAR, ACER and ICER cost analysis techniques are implemented for Strategy 1 (from Scenario A), Strategy 6 (from Scenario B), Strategy 12 (from Scenario C) and Strategy 14 (from Scenario D).

To compare Strategies 1, 6, 12 and 14 using IAR cost analysis technique, it is observed in Figure \ref{S34} and Table \ref{ICERmost} that Strategy 1 has the highest IAR value, followed by Strategy 11, Strategy 6 and Strategy 14. It follows that Strategy 1 (practising physical or social distancing protocols only) is the overall most cost-effective strategy among all the strategies of Scenarios A to D combined as analysed in this work.

Based on the ACER cost analysis technique, and using the results illustrated in Figure \ref{S35} and Table \ref{ICERmost}, it is noted that Strategy 1 has the least ACER value. Strategy 6 is the next strategy with the least ACER value, followed by Strategy 14, then Strategy 11, which has the highest ACER value. Therefore, Strategy 1 is also the most cost-effective strategy among all the 14 control strategies considered in this paper.

\begin{figure}[h!]
	\begin{subfigure}{.5\textwidth}
		\centering
		\includegraphics[width=1\linewidth,  height=2in]{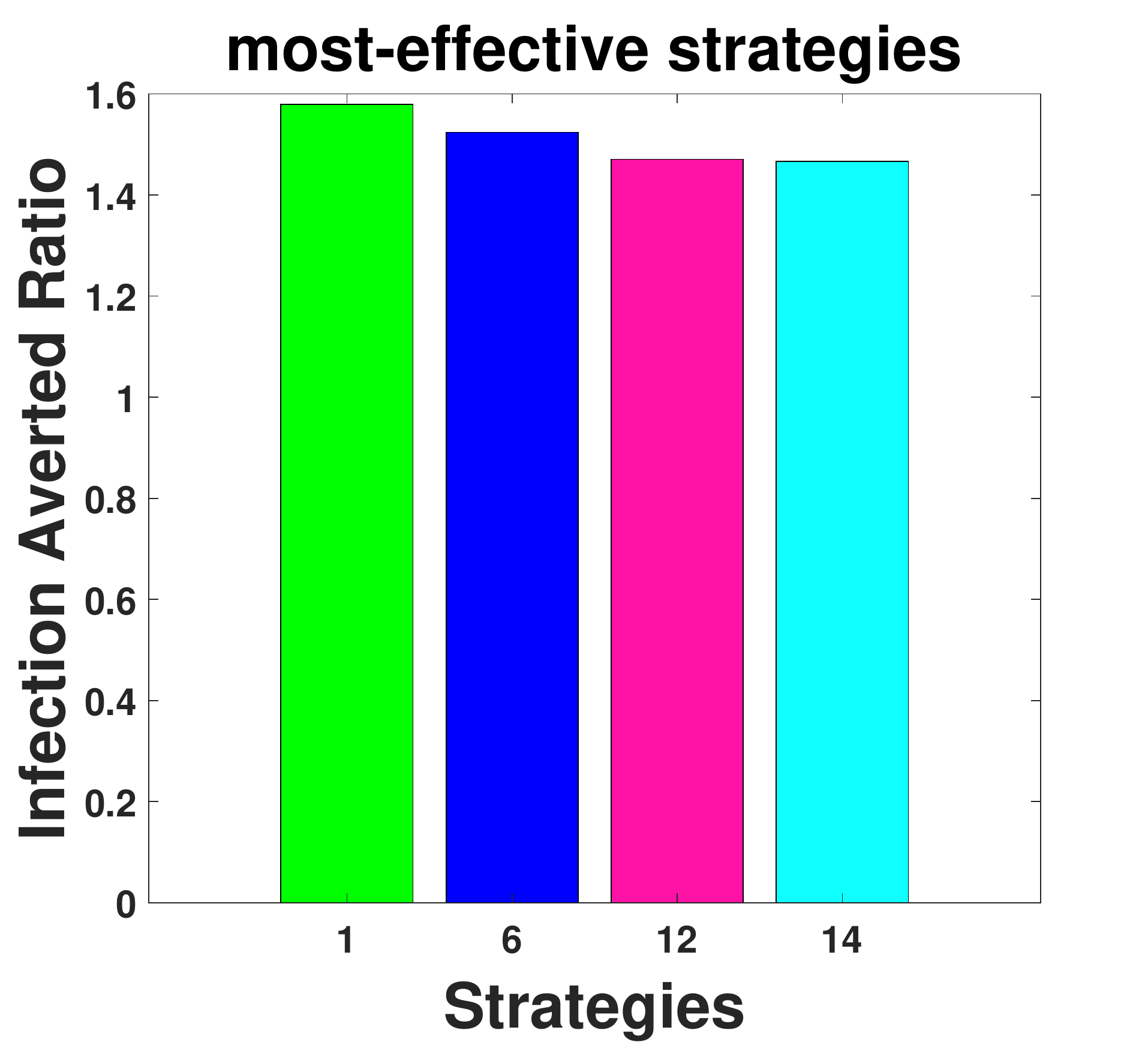}
		\caption{Infection Averted Ratio}\label{S34}
	\end{subfigure}%
	\begin{subfigure}{.5\textwidth}
		\centering
		\includegraphics[width=1\linewidth, height=2in ]{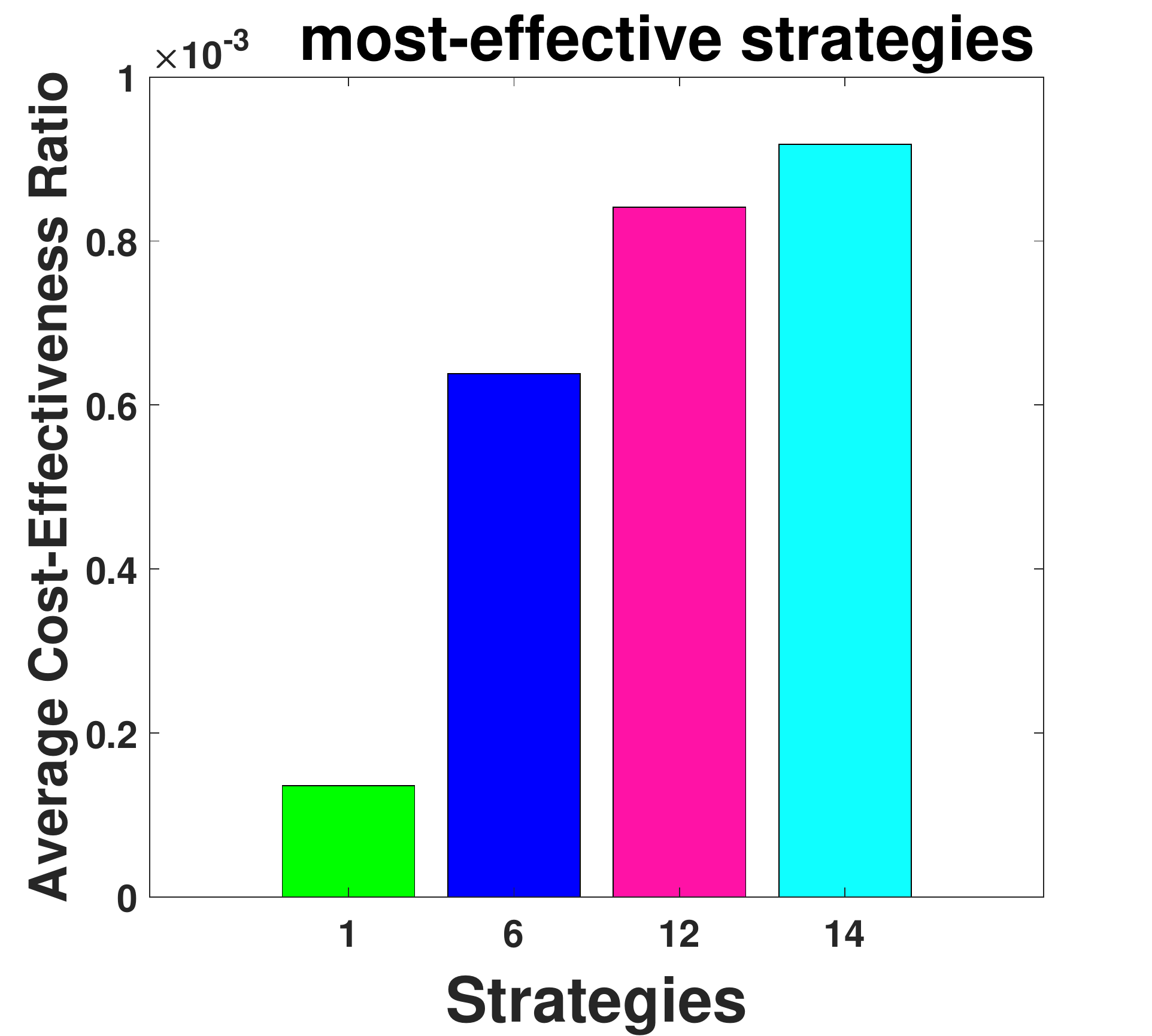}
		\caption{Average Cost-Effectiveness Ratio}\label{S35}
	\end{subfigure}
	\caption{IAR and ACER for the most-effective strategies in Scenarios A--D.}\label{de}
\end{figure}

It remains to compare Strategies 1, 6, 12 and 14 using the ICER cost analysis technique. To do this, the control strategies are ranked in increasing order of their effectiveness according to the total number of infections averted (IA) as given in Table \ref{ICERmost}.

The calculation of ICER is as follows:
\begin{align*}
\mbox{ICER(1)}&=\frac{281.1135}{2.0679\times 10^{6}}= 1.3594\times 10^{-4},\\
\mbox{ICER(6)}&=\frac{1.3487\times 10^3-281.1135}{2.1128\times10^6-2.0679\times10^{6}}= 0.0238,\\
\mbox{ICER(14)}&=\frac{2.0437\times10^3-1.3487\times10^3}{2.2265\times 10^6-2.1128\times10^{6} }= 0.0061.
\end{align*}
Note also that, due the equal number of infection averted by averted by Strategies 14 and 11, the ICER is not compared between these strategies.
The summary of the results is given in Table \ref{ICERmost}.
\begin{table}[h!]
	\centering
	\caption{Incremental cost-effectiveness ratio for the most-effective strategies\label{ICERmost}}
	\begin{tabular}{lccccc}
		\toprule
		{\bf Strategy} & {\bf IA $\times 10^{6}$} & {\bf Cost} & {\bf IAR} & {\bf ACER}&{\bf ICER}  \\ \toprule
		Strategy  1: $u_1(t),$ & $2.0679$ & $281.1135$ & $1.5793$ & $1.3594\times 10^{-4}$&$1.3594\times 10^{-4}$\\
		Strategy 6: $u_1(t), u_3(t)$ & $2.1128$ & $1.3487\times 10^{3}$ & $1.5239$ & $6.3834\times 10^{-4}$&$0.0238$\\
		Strategy 14: $u_1(t), u_2(t),u_3(t), u_4(t)$ & $2.2265$ & $2.0437\times 10^{3}$ & $1.4662$ & $9.1789\times 10^{-4}$& $ 0.0061$\\
		Strategy 12: $u_1(t), u_2(t), u_4(t)$ & $2.2265$ & $1.8726\times 10^{3}$ & $1.4706$ & $8.4107\times 10^{-4}$&$-$\\
		\bottomrule
	\end{tabular}
\end{table}
Table \ref{ICERmost} reveals a cost-saving of $\$ 0.0061$ for Strategy 14 over Strategy 6, following based on the comparison of ICER(14) and ICER(6). The lower ICER value obtained for Strategy 14 indicates that Strategy 6 is dominant, more costly and less effective to implement than Strategy 1. Thus, Strategy 6 is discarded from the set of alternative control interventions. The ICER is then calculated for Strategy 1 and Strategy 14 as iterated below and shown in Table \ref{ICERmost1}.
\begin{align*}
\mbox{ICER(1)}&=\frac{281.1135}{2.0679\times 10^{6}}= 1.3594\times 10^{-4},\\
\mbox{ICER(14)}&=\frac{2.0437\times10^3-281.1135}{2.2265\times 10^6-2.0679\times10^{6} }=0.0111.
\end{align*}
\begin{table}[h!]
	\centering
	\caption{Incremental cost-effectiveness ratio for the most-effective strategies\label{ICERmost1}}
	\begin{tabular}{lccccc}
		\toprule
		{\bf Strategy} & {\bf IA $\times 10^{6}$} & {\bf Cost} & {\bf IAR} & {\bf ACER}&{\bf ICER}  \\ \toprule
		Strategy  1: $u_1(t),$ & $2.0679$ & $281.1135$ & $1.5793$ & $1.3594\times 10^{-4}$&$1.3594\times 10^{-4}$\\
		Strategy 14: $u_1(t), u_2(t),u_3(t), u_4(t)$ & $2.2265$ & $2.0437\times 10^{3}$ & $1.4662$ & $9.1789\times 10^{-4}$& $ 0.0111$\\
	Strategy 12: $u_1(t), u_2(t), u_4(t)$ & $2.2265$ & $1.8726\times 10^{3}$ & $1.4706$ & $8.4107\times 10^{-4}$&$-$\\
		\bottomrule
	\end{tabular}
\end{table}
Table \ref{ICERmost1} reveals a cost-saving of $\$ 1.3594\times 10^{-4}$ for Strategy 1 over Strategy 14 based on the comparison of ICER(1) and ICER(14). The lower ICER value obtained for Strategy 1 indicates that Strategy 14 is dominant, more costly and less effective to implement than Strategy 1. Thus, Strategy 14 is discarded from the set of alternative control interventions.
The ICER is then recalculated for Strategy 1 and Strategy 12 as iterated below and shown in Table \ref{ICERmost2}.
\begin{align*}
\mbox{ICER(1)}&=\frac{281.1135}{2.0679\times 10^{6}}= 1.3594\times 10^{-4},\\
\mbox{ICER(12)}&=\frac{1.8726\times10^3-281.1135}{2.2265\times 10^6-2.0679\times10^{6} }=0.0100.
\end{align*}
\begin{table}[h!]
	\centering
	\caption{Incremental cost-effectiveness ratio for the most-effective strategies\label{ICERmost2}}
	\begin{tabular}{lccccc}
		\toprule
		{\bf Strategy} & {\bf IA $\times 10^{6}$} & {\bf Cost} & {\bf IAR} & {\bf ACER}&{\bf ICER}  \\ \toprule
		Strategy  1: $u_1(t),$ & $2.0679$ & $281.1135$ & $1.5793$ & $1.3594\times 10^{-4}$&$1.3594\times 10^{-4}$\\
		Strategy 12: $u_1(t), u_2(t), u_4(t)$ & $2.2265$ & $1.8726\times 10^{3}$ & $1.4706$ & $8.4107\times 10^{-4}$&$0.0100$\\
		\bottomrule
	\end{tabular}
\end{table}
Table \ref{ICERmost2} reveals a cost-saving of $\$ 1.3594\times 10^{-4}$ for Strategy 1 over Strategy 12 following based on the comparison of ICER(1) and ICER(14). The lower ICER value obtained for Strategy 1 indicates that Strategy 12 is dominant, more costly and less effective to implement than Strategy 1.
Therefore, comparing the strategies in scenarios A-D, we conclude that, Strategy 1 will be the most cost-saving and most effective control intervention in the Kingdom of Saudi Arabia. However, in terms of the infection averted, strategy 6, strategy 11, and strategy 12 and strategy 14 are just as good as strategy 1.

 Hence, from these analyses, we see that when one considers the following controls: $u_1$-practising physical or social distancing protocols; $u_2$-practising personal hygiene by cleaning contaminated surfaces with alcohol-based detergents; $u_3$-practising proper and safety measures by exposed, asymptomatic infected and asymptomatic infected individuals; $u_4$-fumigating schools in all levels of education, sports facilities and commercial areas such as markets and public toilet facilities in Kingdom of Saudi Arabia. $u_1$ (practising physical or social distancing protocols) has the lowest incremental cost-effectiveness and, therefore, gives the optimal cost on a large scale than all the other strategies.
\section{Concluding remarks}\label{ssec:6}
 We formulated an optimal control model for the model proposed in \cite{alqarni2020mathematical}. We used four COVID-19 controls in the absence of vaccination thus, practising physical or social distancing protocols; practising personal hygiene by cleaning contaminated surfaces with alcohol-based detergents; practising proper and safety measures by exposed, asymptomatic infected and asymptomatic infected individuals; and fumigating schools in all levels of education, sports facilities and commercial areas such as markets and public toilet facilities in Kingdom of Saudi Arabia. The implementation of all the control shows that the disease can be reduced when individuals strictly stick to the proposed controls in this work. The efficacy plots in Figure \ref{SD5} shows that the controls become much more effective after 39 days.  We also calculated the infection averted ratio (IAR), average cost-effectiveness ratio (ACER)  and the incremental cost-effectiveness ratio (ICER). We also utilized the cost-minimization analysis when it becomes evident that strategies 11, 12, and 14 had the same number of infection averted.
\section*{CRediT authorship contribution statement}
All authors contributed equally to the development of this work.
\section*{Funding}
This work was funded by the National Natural Science Foundation of China under the Grant number: 12022113, Henry Fok Foundation for Young Teachers (171002).
\section*{Declaration of Competing Interest}
The authors declare that they have no known competing financial interests or personal relationships that could have appeared to influence the findings reported in this paper.

\bibliography{covid19modelingreferences}

\begin{thebibliography}{10}
\expandafter\ifx\csname url\endcsname\relax
  \def\url#1{\texttt{#1}}\fi
\expandafter\ifx\csname urlprefix\endcsname\relax\def\urlprefix{URL }\fi
\expandafter\ifx\csname href\endcsname\relax
  \def\href#1#2{#2} \def\path#1{#1}\fi

\bibitem{piovella2020analytical}
N.~Piovella, Analytical solution of seir model describing the free spread of
  the {COVID-19} pandemic, Chaos, Solitons \& Fractals 140 (2020) 110243.

\bibitem{arino2020simple}
J.~Arino, S.~Portet, A simple model for {COVID-19}, Infectious Disease
  Modelling 5 (2020) 309.

\bibitem{wangping2020extended}
J.~Wangping, H.~Ke, S.~Yang, C.~Wenzhe, W.~Shengshu, Y.~Shanshan, W.~Jianwei,
  K.~Fuyin, T.~Penggang, L.~Jing, et~al., Extended sir prediction of the
  epidemics trend of {COVID-19} in italy and compared with hunan, china,
  Frontiers in medicine 7 (2020) 169.

\bibitem{cui2020dynamic}
Q.~Cui, Z.~Hu, Y.~Li, J.~Han, Z.~Teng, J.~Qian, Dynamic variations of the
  {COVID-19} disease at different quarantine strategies in wuhan and mainland
  china, Journal of Infection and Public Health 13 (2020) 849--855.

\bibitem{hu2020dynamic}
L.~Hu, L.-F. Nie, Dynamic modeling and analysis of {COVID-19} in different
  transmission process and control strategies, Mathematical Methods in the
  Applied Sciences (2020) 1--14.

\bibitem{mushayabasa2020role}
S.~Mushayabasa, E.~T. Ngarakana-Gwasira, J.~Mushanyu, On the role of
  governmental action and individual reaction on {COVID-19} dynamics in south
  africa: A mathematical modelling study, Informatics in Medicine Unlocked 20
  (2020) 100387.

\bibitem{garba2020modeling}
S.~M. Garba, J.~M.-S. Lubuma, B.~Tsanou, Modeling the transmission dynamics of
  the {COVID-19} pandemic in south africa, Mathematical biosciences 328 (2020)
  108441.

\bibitem{fatima2020modeling}
B.~Fatima, G.~Zaman, M.~A. Alqudah, T.~Abdeljawad, Modeling the pandemic trend
  of 2019 coronavirus with optimal control analysis, Results in Physics (2020)
  103660.

\bibitem{ali2020role}
M.~Ali, S.~T.~H. Shah, M.~Imran, A.~Khan, The role of asymptomatic class,
  quarantine and isolation in the transmission of {COVID-19}, Journal of
  Biological Dynamics 14~(1) (2020) 389--408.

\bibitem{lemecha2020optimal}
L.~Lemecha~Obsu, S.~Feyissa~Balcha, Optimal control strategies for the
  transmission risk of {COVID-19}, Journal of biological dynamics 14~(1) (2020)
  590--607.

\bibitem{aldila2020optimal}
D.~Aldila, M.~Z. Ndii, B.~M. Samiadji, Optimal control on {COVID-19}
  eradication program in indonesia under the effect of community awareness.

\bibitem{deressa2020modeling}
C.~T. Deressa, G.~F. Duressa, Modeling and optimal control analysis of
  transmission dynamics of {COVID-19}: The case of ethiopia, Alexandria
  Engineering Journal.

\bibitem{perkins2020optimal}
A.~Perkins, G.~Espana, Optimal control of the {COVID-19} pandemic with
  non-pharmaceutical interventions, Bulletin of Mathematical Biology 82 (2020)
  118.

\bibitem{oud2021fractional}
M.~A.~A. Oud, A.~Ali, H.~Alrabaiah, S.~Ullah, M.~A. Khan, S.~Islam, A
  fractional order mathematical model for {COVID-19} dynamics with quarantine,
  isolation, and environmental viral load, Advances in Difference Equations
  2021~(1) (2021) 1--19.

\bibitem{asamoah2020global}
J.~K.~K. Asamoah, M.~A. Owusu, Z.~Jin, F.~Oduro, A.~Abidemi, E.~O. Gyasi,
  Global stability and cost-effectiveness analysis of {COVID-19} considering
  the impact of the environment: using data from ghana, Chaos, Solitons \&
  Fractals 140 (2020) 110103.
\newblock \href {http://dx.doi.org/https://doi.org/10.1016/j.chaos.2020.110103}
  {\path{doi:https://doi.org/10.1016/j.chaos.2020.110103}}.

\bibitem{alqarni2020mathematical}
M.~S. Alqarni, M.~Alghamdi, T.~Muhammad, A.~S. Alshomrani, M.~A. Khan,
  Mathematical modeling for novel coronavirus ({COVID-19}) and control,
  Numerical Methods for Partial Differential Equations.

\bibitem{seidu2020optimal}
B.~Seidu, {Optimal Strategies for Control of COVID-19: A Mathematical
  Perspective}, Scientifica 2020.

\bibitem{omame2020analysis}
A.~Omame, N.~Sene, I.~Nometa, C.~I. Nwakanma, E.~U. Nwafor, N.~O. Iheonu,
  D.~Okuonghae, {Analysis of COVID-19 and comorbidity co-infection Model with
  Optimal Control}, medRxiv.

\bibitem{asamoah2021sensitivity}
J.~K.~K. Asamoah, Z.~Jin, G.-Q. Sun, B.~Seidu, E.~Yankson, A.~Abidemi,
  F.~Oduro, S.~E. Moore, E.~Okyere, Sensitivity assessment and optimal economic
  evaluation of a new {COVID-19} compartmental epidemic model with control
  interventions, Chaos, Solitons \& Fractals 146 (2021) 110885.
\newblock \href {http://dx.doi.org/https://doi.org/10.1016/j.chaos.2021.110885}
  {\path{doi:https://doi.org/10.1016/j.chaos.2021.110885}}.

\bibitem{okyere2020analysis}
E.~Okyere, S.~Olaniyi, E.~Bonyah, Analysis of {Zika} virus dynamics with sexual
  transmission route using multiple optimal controls, Scientific African 9
  (2020) e00532.

\bibitem{panja2019optimal}
P.~Panja, Optimal control analysis of a cholera epidemic model, Biophysical
  Reviews and Letters 14~(01) (2019) 27--48.

\bibitem{agusto2014optimal}
F.~Agusto, A.~Adekunle, {Optimal control of a two-strain tuberculosis-HIV/AIDS
  co-infection model}, Biosystems 119 (2014) 20--44.

\bibitem{agusto2019optimal}
F.~Agusto, M.~Leite, Optimal control and cost-effective analysis of the 2017
  meningitis outbreak in nigeria, Infectious Disease Modelling 4 (2019)
  161--187.

\bibitem{berhe2020optimal}
H.~W. Berhe, {Optimal Control Strategies and Cost-effectiveness Analysis
  Applied to Real Data of Cholera Outbreak in Ethiopia’s Oromia Region},
  Chaos, Solitons \& Fractals 138 (2020) 109933.

\bibitem{momoh2018optimal}
A.~A. Momoh, A.~F{\"u}genschuh, Optimal control of intervention strategies and
  cost effectiveness analysis for a {Zika} virus model, Operations Research for
  Health Care 18 (2018) 99--111.

\bibitem{olaniyi2020modelling}
S.~Olaniyi, K.~Okosun, S.~Adesanya, R.~Lebelo, Modelling malaria dynamics with
  partial immunity and protected travellers: optimal control and
  cost-effectiveness analysis, Journal of Biological Dynamics 14~(1) (2020)
  90--115.

\bibitem{tilahun2017modelling}
G.~T. Tilahun, O.~D. Makinde, D.~Malonza, Modelling and optimal control of
  pneumonia disease with cost-effective strategies, Journal of Biological
  Dynamics 11~(sup2) (2017) 400--426.

\bibitem{berhe2019co}
H.~W. Berhe, O.~D. Makinde, D.~M. Theuri, Co-dynamics of measles and dysentery
  diarrhea diseases with optimal control and cost-effectiveness analysis,
  Applied Mathematics and Computation 347 (2019) 903--921.

\bibitem{berhe2018optimal}
H.~W. Berhe, O.~D. Makinde, D.~M. Theuri, Optimal control and
  cost-effectiveness analysis for dysentery epidemic model, Applied Mathematics
  and Information Sciences 12 (2018) 1183--1195.

\bibitem{abidemi2020optimal}
A.~Abidemi, N.~A.~B. Aziz, Optimal control strategies for dengue fever spread
  in {Johor, Malaysia}, Computer Methods and Programs in Biomedicine (2020)
  105585.

\bibitem{asamoah2021non}
J.~K.~K. Asamoah, Z.~Jin, G.-Q. Sun, Non-seasonal and seasonal relapse model
  for {Q} fever disease with comprehensive cost-effectiveness analysis, Results
  in Physics (2021) 103889\href
  {http://dx.doi.org/https://doi.org/10.1016/j.rinp.2021.103889}
  {\path{doi:https://doi.org/10.1016/j.rinp.2021.103889}}.

\bibitem{seidu2021mathematical}
B.~Seidu, O.~Makinde, C.~S. Bornaa, Mathematical analysis of an industrial
  {HIV/AIDS} model that incorporates carefree attitude towards sex, Acta
  Biotheoretica (2021) 1--20.

\bibitem{khan2021robust}
M.~A. Khan, S.~Ullah, S.~Kumar, A robust study on 2019-ncov outbreaks through
  non-singular derivative, The European Physical Journal Plus 136~(2) (2021)
  1--20.
\newblock \href
  {http://dx.doi.org/https://doi.org/10.1140/epjp/s13360-021-01159-8}
  {\path{doi:https://doi.org/10.1140/epjp/s13360-021-01159-8}}.

\bibitem{abidemi2021impact}
A.~Abidemi, Z.~M. Zainuddin, N.~A.~B. Aziz, {Impact of control interventions on
  COVID-19 population dynamics in Malaysia: a mathematical study}, The European
  Physical Journal Plus 136~(2) (2021) 1--35.

\bibitem{iddrisu2021predictive}
A.-K. Iddrisu, E.~A. Amikiya, D.~Otoo, {A predictive model for daily cumulative
  COVID-19 cases in Ghana}, F1000Research 10~(343) (2021) 343.

\bibitem{acheampong2021modelling}
E.~Acheampong, E.~Okyere, S.~Iddi, J.~H. Bonney, J.~A. Wattis, R.~L. Gomes,
  {Modelling COVID-19 Transmission Dynamics in Ghana}, arXiv preprint
  arXiv:2102.02984.

\bibitem{otoo2021estimating}
D.~Otoo, E.~K. Donkoh, J.~A. Kessie, {Estimating the Basic Reproductive Number
  of COVID-19 Cases in Ghana}, European Journal of Pure and Applied Mathematics
  14~(1) (2021) 135--148.

\bibitem{asamoah2020mathematical}
J.~K.~K. Asamoah, C.~Bornaa, B.~Seidu, Z.~Jin, Mathematical analysis of the
  effects of controls on transmission dynamics of {SARS-CoV-2}, Alexandria
  Engineering Journal 59~(6) (2020) 5069--5078.
\newblock \href {http://dx.doi.org/https://doi.org/10.1016/j.aej.2020.09.033}
  {\path{doi:https://doi.org/10.1016/j.aej.2020.09.033}}.

\bibitem{khan2020dynamics}
M.~A. Khan, A.~Atangana, E.~Alzahrani, et~al., The dynamics of {COVID-19} with
  quarantined and isolation, Advances in Difference Equations 2020~(1) (2020)
  1--22.
\newblock \href {http://dx.doi.org/https://doi.org/10.1186/s13662-020-02882-9}
  {\path{doi:https://doi.org/10.1186/s13662-020-02882-9}}.

\bibitem{khan2020modeling}
M.~A. Khan, A.~Atangana, Modeling the dynamics of novel coronavirus
  {(2019-nCov)} with fractional derivative, Alexandria Engineering Journal
  59~(4) (2020) 2379--2389.
\newblock \href {http://dx.doi.org/https://doi.org/10.1016/j.aej.2020.02.033}
  {\path{doi:https://doi.org/10.1016/j.aej.2020.02.033}}.

\bibitem{moore2021global}
S.~E. Moore, H.~L.~N. Bamen, J.~K.~K. Asamoah, O.~Menoukeu-Pamen, Z.~Jin,
  {Global Stability and Sensitivity Assessment of COVID-19 with Timely and
  Delayed Diagnosis in Ghana}.

\bibitem{ngonghala2020mathematical}
C.~N. Ngonghala, E.~Iboi, S.~Eikenberry, M.~Scotch, C.~R. MacIntyre, M.~H.
  Bonds, A.~B. Gumel, Mathematical assessment of the impact of
  non-pharmaceutical interventions on curtailing the 2019 novel coronavirus,
  Mathematical biosciences 325 (2020) 108364.
\newblock \href {http://dx.doi.org/https://doi.org/10.1016/j.mbs.2020.108364}
  {\path{doi:https://doi.org/10.1016/j.mbs.2020.108364}}.

\bibitem{sun2020transmission}
G.-Q. Sun, S.-F. Wang, M.-T. Li, L.~Li, J.~Zhang, W.~Zhang, Z.~Jin, G.-L. Feng,
  Transmission dynamics of {COVID-19} in wuhan, china: effects of lockdown and
  medical resources, Nonlinear Dynamics 101~(3) (2020) 1981--1993.
\newblock \href {http://dx.doi.org/https://doi.org/10.1007/s11071-020-05770-9}
  {\path{doi:https://doi.org/10.1007/s11071-020-05770-9}}.

\bibitem{eikenberry2020mask}
S.~E. Eikenberry, M.~Mancuso, E.~Iboi, T.~Phan, K.~Eikenberry, Y.~Kuang,
  E.~Kostelich, A.~B. Gumel, To mask or not to mask: Modeling the potential for
  face mask use by the general public to curtail the covid-19 pandemic,
  Infectious Disease Modelling 5 (2020) 293--308.
\newblock \href {http://dx.doi.org/https://doi.org/10.1016/j.idm.2020.04.001}
  {\path{doi:https://doi.org/10.1016/j.idm.2020.04.001}}.

\bibitem{lopez2021modified}
L.~L{\'o}pez, X.~Rodo, A modified seir model to predict the covid-19 outbreak
  in spain and italy: simulating control scenarios and multi-scale epidemics,
  Results in Physics 21 (2021) 103746.
\newblock \href {http://dx.doi.org/https://doi.org/10.1016/j.rinp.2020.103746}
  {\path{doi:https://doi.org/10.1016/j.rinp.2020.103746}}.

\bibitem{abay2021mathematical}
A.~Abay, H.~W. Berhe, H.~A. Atsbaha, Mathematical modelling and analysis of
  {COVID-19} epidemic and predicting its future situation in ethiopia., Results
  in Physics (2021) 103853\href
  {http://dx.doi.org/https://doi.org/10.1016/j.rinp.2021.103853}
  {\path{doi:https://doi.org/10.1016/j.rinp.2021.103853}}.

\bibitem{bassey2021global}
B.~E. Bassey, J.~U. Atsu, Global stability analysis of the role of
  multi-therapies and non-pharmaceutical treatment protocols for {COVID-19}
  pandemic, Chaos, Solitons \& Fractals 143 (2021) 110574.

\bibitem{agusto2013optimal}
F.~Agusto, Optimal isolation control strategies and cost-effectiveness analysis
  of a two-strain avian influenza model, Biosystems 113~(3) (2013) 155--164.

\bibitem{abidemi2019global}
A.~Abidemi, R.~Ahmad, N.~A.~B. Aziz, {Global Stability and Optimal Control of
  Dengue with Two Coexisting Virus Serotypes}, MATEMATIKA: Malaysian Journal of
  Industrial and Applied Mathematics 35~(4) (2019) 149--170.

\bibitem{asamoah2017modelling}
J.~K.~K. Asamoah, F.~T. Oduro, E.~Bonyah, B.~Seidu, Modelling of rabies
  transmission dynamics using optimal control analysis, Journal of Applied
  Mathematics 2017.
\newblock \href {http://dx.doi.org/https://doi.org/10.1155/2017/2451237}
  {\path{doi:https://doi.org/10.1155/2017/2451237}}.

\bibitem{olaniyi2020mathematical}
S.~Olaniyi, O.~Obabiyi, K.~Okosun, A.~Oladipo, S.~Adewale, Mathematical
  modelling and optimal cost-effective control of {COVID-19} transmission
  dynamics, The European Physical Journal Plus 135~(11) (2020) 1--20.

\bibitem{asamoah2020deterministic}
J.~K.~K. Asamoah, Z.~Jin, G.-Q. Sun, M.~Y. Li, {A Deterministic Model for Q
  Fever Transmission Dynamics within Dairy Cattle Herds: Using Sensitivity
  Analysis and Optimal Controls}, Computational and mathematical methods in
  medicine 2020.
\newblock \href {http://dx.doi.org/https://doi.org/10.1155/2020/6820608}
  {\path{doi:https://doi.org/10.1155/2020/6820608}}.

\bibitem{asamoah2020backward}
J.~K.~K. Asamoah, F.~Nyabadza, Z.~Jin, E.~Bonyah, M.~A. Khan, M.~Y. Li,
  T.~Hayat, Backward bifurcation and sensitivity analysis for bacterial
  meningitis transmission dynamics with a nonlinear recovery rate, Chaos,
  Solitons \& Fractals 140 (2020) 110237.
\newblock \href {http://dx.doi.org/https://doi.org/10.1016/j.chaos.2020.110237}
  {\path{doi:https://doi.org/10.1016/j.chaos.2020.110237}}.

\bibitem{asamoah2018mathematical}
J.~K.~K. Asamoah, F.~Nyabadza, B.~Seidu, M.~Chand, H.~Dutta, Mathematical
  modelling of bacterial meningitis transmission dynamics with control
  measures, Computational and mathematical methods in medicine 2018.
\newblock \href {http://dx.doi.org/https://doi.org/10.1155/2018/2657461}
  {\path{doi:https://doi.org/10.1155/2018/2657461}}.

\bibitem{rector2005principles}
C.~R. Rector, C.~S, D.~J, Principles of Optimization Theory, Narosa Publishing
  House, New Delhi, 2005.

\bibitem{romero2018optimal}
J.~P. Romero-Leiton, J.~M. Montoya~Aguilar, E.~Ibarg{\"u}en-Mondrag{\'o}n, An
  optimal control problem applied to malaria disease in {Colombia}, Applied
  Mathematical Sciences 12~(6) (2018) 279--292.

\bibitem{pontryagin1962mathematical}
L.~Pontryagin, V.~Boltyanskii, R.~Gamkrelidze, E.~Mishchenko, The Mathematical
  Theory of Optimal Processes, Wiley, NY, 1962.

\bibitem{agusto2017optimal}
F.~B. Agusto, I.~M. ELmojtaba, Optimal control and cost-effective analysis of
  malaria/visceral leishmaniasis co-infection, PLoS One 12~(2) (2017) e0171102.

\end{thebibliography}



\end{document}